\documentclass{amsart}
\usepackage{amssymb,stmaryrd,color,xypic,mathdots,graphicx}

\usepackage{hyperref}

\theoremstyle{theorem}
\newtheorem{theorem}{Theorem}[section]
\newtheorem{lemma}[theorem]{Lemma}
\newtheorem{proposition}[theorem]{Proposition}
\newtheorem{corollary}[theorem]{Corollary}

\theoremstyle{definition}
\newtheorem{definition}[theorem]{Definition}
\newtheorem{remark}[theorem]{Remark}
\newtheorem*{Notation}{Notation}

\begin{document}

\title[On the failure of the first \v{C}ech homotopy group]{\large O\lowercase{n the failure of the first} \v{C}\lowercase{ech homotopy group to register geometrically relevant fundamental group elements}}

\author{Jeremy Brazas}

\address{Department of Mathematics, West Chester University of Pennsylvania, West Chester, PA 19383, USA}

\email{jbrazas@wcupa.edu}

\author{Hanspeter Fischer}

\address{Department of Mathematical Sciences, Ball State University, Muncie, IN 47306, USA}

\email{hfischer@bsu.edu}



\date{\today}

\begin{abstract}
We construct a space $\mathbb{P}$ for which the canonical homomorphism $\pi_1(\mathbb{P},p) \rightarrow \check{\pi}_1(\mathbb{P},p)$ from the fundamental group to the first \v{C}ech homotopy group is not injective, although it has all of the following properties: (1) $\mathbb{P}\setminus\{p\}$ is a 2-manifold with connected non-compact boundary; (2) $\mathbb{P}$ is connected and locally path connected; (3) $\mathbb{P}$ is strongly homotopically Hausdorff; (4) $\mathbb{P}$ is  homotopically path Hausdorff; (5) $\mathbb{P}$ is  1-UV$_0$; (6) $\mathbb{P}$ admits a simply connected generalized covering space with monodromies between fibers that have discrete graphs; (7) $\pi_1(\mathbb{P},p)$ naturally injects into the inverse limit of finitely generated free monoids otherwise associated with the Hawaiian Earring; (8) $\pi_1(\mathbb{P},p)$ is locally free.
\end{abstract}

\maketitle

\section{Introduction}

The geometric significance of the elements of the fundamental group $\pi_1(X,x)$ of a connected and locally path-connected metric space $X$ is most prominently on display in the context of a simply connected covering space (if it exists) where these elements comprise the group of deck transformations of the covering projection. In this situation, the canonical homomorphism  $\pi_1(X,x) \rightarrow \check{\pi}_1(X,x)$ to the first \v{C}ech homotopy group (also called the first shape group \cite{MS}) is an isomorphism~\cite{FZ2007}.

In fact, as long as all fundamental group elements are accounted for, that is, if $\pi_1(X,x) \hookrightarrow \check{\pi}_1(X,x)$ is injective, the standard covering construction yields a {\em generalized} covering projection $p:\widetilde{X}\rightarrow X$ with connected, locally path-connected and simply connected $\widetilde{X}$ \cite{FZ2007}. It is characterized by the usual unique lifting property and we have $\pi_1(X,x)\cong Aut(\widetilde{X} \stackrel{p}{\rightarrow} X)$. Examples of spaces for which $\pi_1(X,x) \hookrightarrow \check{\pi}_1(X,x)$ is injective include all one-dimensional separable metric spaces \cite{CF1959b,EK}, all planar spaces \cite{FZ2005}, the Pontryagin sphere, the Pontryagin surface $\Pi_2$, and similar inverse limits of higher-dimensional manifolds \cite{FGu}.

Several weaker properties that quantify the geometric relevance of the elements of $\pi_1(X,x)$ can be found in the literature. The {\em strongly homotopically Hausdorff} property, for example, stipulates that for every essential loop in $X$ there should be a limit to how small it can be made at a particular point by a free homotopy \cite{CMRZZ}. The {\em homotopically path Hausdorff} property, on the other hand, calls for $\pi_1(X,x)$ to be T$_1$ in the quotient topology induced by the compact-open topology on the loop space $\Omega(X,x)$ \cite{BFa}. Both of these properties are implied by $\pi_1(X,x) \hookrightarrow \check{\pi}_1(X,x)$ being injective \cite{FRVZ}.

Then there are properties that guarantee, in and of themselves, the existence of a simply connected generalized covering space. These include
the homotopically path Hausdorff property above and the {\em 1-UV$_0$} property, which requires small null-homotopic loops to contract via small homotopies \cite{BFi,FRVZ}.

Even if a generalized covering projection with simply connected domain exists, it might not be a fibration and the monodromies between fibers might not be continuous \cite{FGa}. (Such is the case for the Hawaiian Earring.) However, for all one-dimensional spaces and for all planar spaces, these  monodromies have {\em discrete graphs}; a fact implicitly used in the work of Eda \cite{E2002,E2010} and Conner-Kent \cite{CK}. If any two spaces with this property (cf.\@ Definition~\ref{DefDM}) are homotopy equivalent, then their respective wild sets (points at which they are not semilocally simply connected) are homeomorphic (Theorem~\ref{utility}).

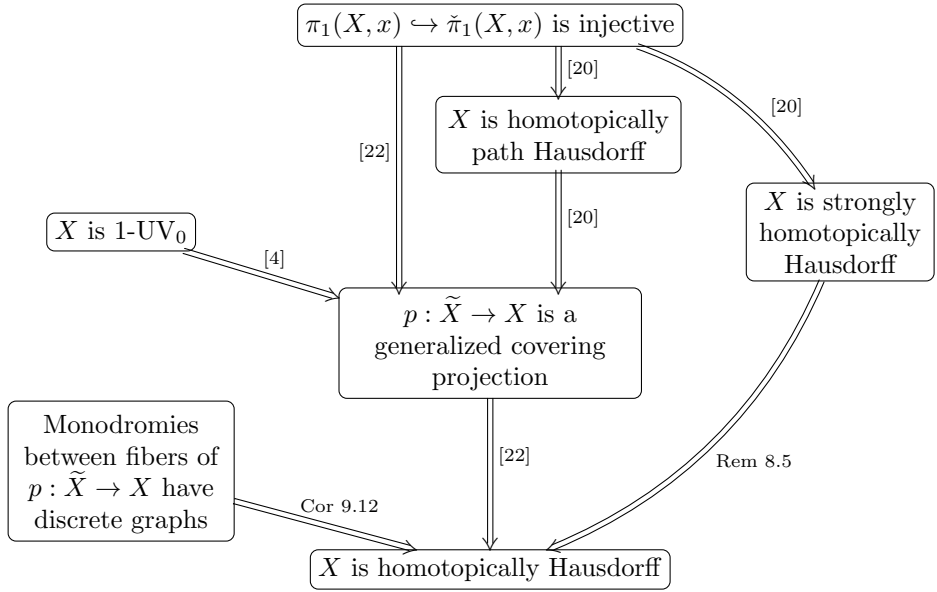
\begin{figure}
\xymatrix@R=7pt{
  &*+[F-:<3pt>]{\mbox{$\pi_1(X,x)\hookrightarrow \check{\pi}_1(X,x)$ is injective}}
  \ar@{=>}[]!<-8ex,0ex>;[dddd]!<-8ex,0ex>_(0.4){\text{[22]}}
 \ar@{=>}[]!<6ex,0ex>;[dd]!<6ex,0ex>^(0.4){\text{[20]}}
 \ar@{=>}@/^2pc/[dddr]^(0.7){\text{[20]}}
  \\ &&\\
&*+!<-6ex,0ex>[F-:<3pt>]{\parbox{1.2in}{\center $X$ is homotopically path Hausdorff}}\ar@{=>}[]!<6ex,0ex>;[dd]!<6ex,0ex>^(0.4){\text{[20]}}& \\
*+[F-:<3pt>]{\mbox{$X$ is 1-UV$_0$}} \ar@{=>}[dr]^(0.4){\text{[4]}} & & *+[F-:<3pt>]{\parbox{0.9in}{\center \vspace{-2pt} $X$ is strongly homotopically Hausdorff}} \ar@{=>}@/^2pc/[dddl]^{\text{Rem 8.5} } \\
& *+[F-:<3pt>]{\parbox{1.5in}{ \center \vspace{-2pt} $p:\widetilde{X}\rightarrow X$ is a generalized covering projection}} \ar@{=>}[dd]^(0.5){\text{[22]}}& \\
 *+[F-:<3pt>]{\parbox{1.1in}{\center Monodromies between fibers of $p:\widetilde{X}\rightarrow X$ have discrete graphs}} \ar@{=>}[dr]^{\text{\hspace{10pt} Cor 9.12}}& & \\
 & *+[F-:<3pt>]{\parbox{1.8in}{$X$ is homotopically Hausdorff}}& \\
}
\caption{Relationships between some local  properties of $\pi_1(X,x)$}
\end{figure}

As far as the algebraic structure of fundamental groups of low-dimensional spaces is concerned, we recall that the fundamental group of an arbitrary planar Peano continuum (not necessarily homotopy equivalent to a one-dimensional space) is isomorphic to a subgroup of the fundamental group of some one-dimensional planar Peano continuum \cite{CC}.
In turn,
 fundamental groups of one-dimensional path-connected separable metric spaces are locally free \cite{CF1959b} and, in the compact case, have structures similar to that of the Hawaiian Earring $\mathbb{H}$, where the injective function $\pi_1(\mathbb{H})\hookrightarrow \check{\pi}_1(\mathbb{H})$ factors through the limit $\mathcal M$  of an inverse system ${\mathcal M}_1\stackrel{R_1}{\longleftarrow} {\mathcal M}_2\stackrel{R_2}{\longleftarrow}{\mathcal M}_3\stackrel{R_3}{\longleftarrow} \cdots$ of  free monoids ${\mathcal M}_n$ on $\{\ell_1^{\pm 1}, \ell_2^{\pm 1}, \cdots, \ell_n^{\pm 1}\}$ with $R_{n-1}$ deleting the letters $\ell_n$ and $\ell_n^{-1}$ from every word \cite{DTW,FZ2013b}. \vspace{5pt}

This raises the following
\vspace{5pt}

\noindent {\bf Question:} {\em Is there a space $X$ for which $\pi_1(X)\rightarrow \check{\pi}_1(X)$ is \underline{not} injective, but with all of the  other properties discussed above:  low-dimensional, strongly homotopically Hausdorff, homotopically path Hausdorff, 1-UV$_0$, admits a generalized universal covering $p:\widetilde{X}\rightarrow X$ whose monodromies between fibers have discrete graphs, admits a natural injective function $\pi_1(X)\rightarrow {\mathcal M}$, and has locally free fundamental group?}
\vspace{5pt}

We present a relatively simple and prototypical construction of a two-dimensional space that yields a positive answer to this question.

Specifically, in Section~\ref{construction}, we define the space $\mathbb{P}$ mentioned in the abstract, by attaching countably many ``pairs of pants'' to the Hawaiian Earring $\mathbb{H}$, and identify its fundamental group as a direct limit of groups each isomorphic to $\pi_1(\mathbb{H})$ with injective bonding homomorphisms. In particular, $\pi_1(\mathbb{P})$ is locally free (Proposition~\ref{lf}).

We show that $\pi_1(\mathbb{P})\rightarrow \check{\pi}_1(\mathbb{P})$ is not injective (Theorem~\ref{thm:notinj}), but that
$\mathbb{P}$ is both strongly homotopically Hausdorff (Theorem~\ref{SHHthm}) and homotopically path Hausdorff (Theorem~\ref{HPH}); as far as the authors know, it is the first such example (Remarks~\ref{Z'} and \ref{Y'}). The proof of the latter property hinges on the fact that $\pi_1(\mathbb{P})$ naturally injects into the inverse limit of monoids associated with the Hawaiian Earring\linebreak (Theorem~\ref{inj}), despite $\pi_1(\mathbb{P})$ not being isomorphic to a subgroup of an inverse limit of free groups (Remark~\ref{noalgebra}). Moreover, we show that $\mathbb{P}$ is 1-UV$_0$ (Theorem~\ref{UVThm}) and that $\pi_1(\mathbb{P})$ is locally free (Proposition~\ref{lf}).

After a brief review of generalized covering space theory (Section~\ref{review}) we show that the monodromies for the simply connected generalized covering space of $\mathbb{P}$ have discrete graphs (Theorem~\ref{PhasDMP}). We also discuss some general aspects of this  property, such as its relationship to the homotopically Hausdorff property relative to a subgroup of the fundamental group (Proposition~\ref{DMP->HH}) and its impact on the stability of wild subsets under homotopy equivalence (Theorem~\ref{utility}).

\section{The Hawaiian pants $\mathbb{P}$}\label{construction}

Let $C_n\subseteq \mathbb{R}^2$ be the circle of radius $\frac{1}{n}$ centered at $(\frac{1}{n},0)$ and let $\mathbb{H}=\bigcup_{i=1}^\infty C_i\subseteq \mathbb{R}^2$ be the usual Hawaiian Earring with basepoint $b_0=(0,0)$.  Define $\ell_n:[0,1]\rightarrow C_n$ by $\ell_n(t)=(\frac{1}{n}(1-\cos 2\pi t), \frac{1}{n}\sin 2\pi t))$.

 For  $n\in \mathbb{N}$, let $D_{n,1}$ and $D_{n,2}$ be two disjoint disks in the interior $D_n^\circ$ of a disk $D_n\subseteq \mathbb{R}^2$ and consider the ``pair of pants'' $P_n=D_n\setminus (D^\circ_{n,1}\cup D^\circ_{n,2})$. Let $\alpha_n, \beta_n, \gamma_n:[0,1]\rightarrow P_n$ be parametrizations of the boundaries $\partial D_n$, $\partial D_{n,1}$,  and $\partial D_{n,2}$, respectively, with clockwise orientation.

  Let $\mathbb{P}$ be the space obtained from $\mathbb{H}$ by attaching all $P_n$ via maps $f_n:\partial P_n\rightarrow \mathbb{H}$  such that $f_n\circ\alpha_n=\ell_{n}$, $f_n\circ\beta_n=\ell_{2n}$ and $f_n\circ\gamma_n=\ell_{2n+1}$. That is,  we put $\mathbb{P}=\mathbb{H}\cup_{f} \left(\coprod_{n\in \mathbb{N}} P_n\right)$ where $f:\coprod_{n\in\mathbb{N}}\partial P_n\rightarrow \mathbb{H}$ and $f|_{\partial P_n}=f_n$. We refer to $\mathbb{P}$ as the ``Hawaiian Pants''. (See Figure~\ref{pants}.)
\begin{figure}
\hspace{.4in} \includegraphics[scale=1]{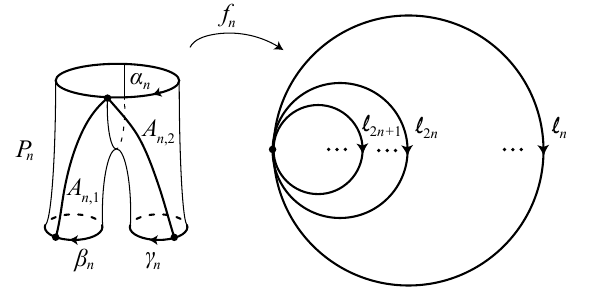}
\caption{\label{pants} Attaching pairs of pants to the Hawaiian Earring}
\end{figure}

We observe that exactly one pair of pants is attached to  $C_1$, namely $P_1$ via $f_1|_{\partial D_1}$, and that for each $n\geqslant 2$, exactly two pairs of pants are attached to $C_n$, namely $P_n$ via $f_n|_{\partial D_n}$, and either $P_{n/2}$ via $f_{n/2}|_{\partial D_{n/2,1}}$ (if $n$ is even) or $P_{(n-1)/2}$ via $f_{(n-1)/2}|_{\partial D_{(n-1)/2,2}}$ (if $n$ is odd). In particular:

\begin{proposition} \hspace{1in}
\begin{itemize}
\item[(a)] $\mathbb{P}\setminus \{b_0\}$ is a 2-manifold with boundary $C_1\setminus \{b_0\}$.
\item[(b)] $\mathbb{P}$ is connected and locally path connected.
\end{itemize}
\end{proposition}

  We have $f^{-1}_n(b_0)=\{a_n,b_n,c_n\}$ with $a_n\in \partial D_n$, $b_n\in \partial D_{n,1}$ and $c_n\in \partial D_{n,2}$.
  Choose arcs $A_{n,1}, A_{n,2}\subseteq P_n$ such that $A_{n,1}\cap \partial P_n=\partial A_{n,1}=\{a_n,b_n\}$, $A_{n,2}\cap \partial P_n=\partial A_{n,2}=\{a_n,c_n\}$, and $A_{n,1}\cap A_{n,2}=\{a_n\}$, configured as in Figure~\ref{pants}. Let $B_{n,j}$ be the image of $A_{n,j}$ in $\mathbb{P}$ when attaching $P_n$. Then $B_{n,j}$ is homeomorphic to a circle. We regard $P^\circ_n=D^\circ_n\setminus(D_{n,1}\cup D_{n,2})$ as a subspace of $\mathbb{P}$.

Let $\mathbb{H}_n=\bigcup_{i=n}^\infty C_i$ and let $\mathbb{P}_n$ be the space obtained from $\mathbb{H}_n$ by attaching  $P_n,P_{n+1},P_{n+2},\dots$. Define $\mathbb{P}_n^+=\bigcup_{i=1}^{n-1} (B_{i,1}\cup B_{i,2})\cup \mathbb{P}_n$ for $n\geqslant 2$ and put $\mathbb{P}_1^+=\mathbb{P}_1=\mathbb{P}$. Then $\mathbb{P}=\mathbb{P}^+_1 \supseteq \mathbb{P}^+_2 \supseteq  \mathbb{P}^+_3 \supseteq \cdots$.

    Since each $P_n$ deformation retracts onto $\partial D_{n,1}\cup A_{n,1}\cup A_{n,2}\cup \partial D_{n,2}$, there are  deformation retractions $\phi_n:\mathbb{P}^+_n\times [0,1] \rightarrow \mathbb{P}^+_n$ such that
    \begin{itemize}
    \item[(i)] for all $p\in \mathbb{P}^+_n$, we have $\phi_n(p,0)=p$;
      \item[(ii)] for all $p\in \mathbb{P}^+_n$, we have $\phi_n(p,1)\in \mathbb{P}^+_{n+1}$;
 \item[(iii)] for all $p\in \mathbb{P}^+_{n+1}$ and all $t\in [0,1]$, we have $\phi_n(p,t)=p$;
 \item[(iv)] for all $p\in \mathbb{P}^+_n\setminus \mathbb{P}^+_{n+1}$ and all $t\in[0,1]$, we have that $\phi_n(p,t)$ lies in the image of $P_n$ in the quotient $\mathbb{P}^+_n$.
    \end{itemize}
Put $\mathbb{H}_{n,k}=\bigcup_{i=n}^k C_i$ (where $\mathbb{H}_{n,k}=\emptyset$ for $k<n$) and $\mathbb{H}_{n,k}^+=\bigcup_{i=1}^{n-1}(B_{i,1}\cup B_{i,2})\cup \mathbb{H}_{n,k}$. Likewise, put $\mathbb{H}_n^+=\bigcup_{i=1}^{n-1}(B_{i,1}\cup B_{i,2})\cup \mathbb{H}_n$. Note that $\mathbb{H}^+_{n,n-1}= \bigcup_{i=1}^{n-1}(B_{i,1}\cup B_{i,2})$ and  $\mathbb{H}_1^+=\mathbb{H}$.

 Defining $d_n(p)=\phi_n(p,1)$ and letting $r_{n,k}:\mathbb{H}^+_{n}\rightarrow \mathbb{H}^+_{n,k}$ denote the canonical retractions with $r_{n,k}(\bigcup_{i=k+1}^\infty C_i)=\{b_0\}$, we have the following commutative diagrams:

\[
\xymatrix{\mathbb{P}^+_1 \ar[r]^{d_1} & \mathbb{P}^+_2 \ar[r]^{d_2} & \cdots \ar[r]^{d_{n-1}} & \mathbb{P}^+_n \ar[r]^{d_n} & \mathbb{P}^+_{n+1} \ar[r]^{d_{n+1}} & \cdots\\
\mathbb{H}^+_1 \ar[r]^{d_1} \ar@{^{(}->}[u] \ar[d]^{r_{1,k}} & \mathbb{H}^+_2 \ar[r]^{d_2}  \ar@{^{(}->}[u] \ar[d]^{r_{2,k}} & \cdots \ar[r]^{d_{n-1}}   & \mathbb{H}^+_n \ar[r]^{d_n} \ar@{^{(}->}[u] \ar[d]^{r_{n,k}} & \mathbb{H}^+_{n+1} \ar[r]^{d_{n+1}} \ar@{^{(}->}[u] \ar[d]^{r_{n+1,k}} & \cdots\\
\mathbb{H}^+_{1,k} \ar[r]^{d_1} & \mathbb{H}^+_{2,k} \ar[r]^{d_2}  & \cdots \ar[r]^{d_{n-1}} & \mathbb{H}^+_{n,k} \ar[r]^{d_n} & \mathbb{H}^+_{n+1,k} &  (k\geqslant 2n+1)
}
\]

\vspace{10pt}

We may assume that there are parametrizations $\rho_{n,j}:([0,1],\{0,1\})\rightarrow (B_{n,j},b_0)$ such that \begin{equation}\label{peel}
d_n\circ \ell_n=\rho_{n,1} \cdot \ell_{2n} \cdot \rho_{n,1}^- \cdot \rho_{n,2}\cdot \ell_{2n+1}\cdot \rho_{n,2}^-.
\end{equation}
Here, ``$\;\cdot\;$'' denotes the usual concatenation of paths and $\rho_{n,j}^-$ denotes
the reverse path of $\rho_{n,j}$, given by $\rho_{n,j}^-(t)=\rho_{n,j}(1-t)$. Taking path homotopy classes, and noting that $[\ell_n]=[d_n\circ \ell_n]\in \pi_1(\mathbb{P},b_0)$, we have
\begin{equation}\label{relation}
[\ell_n]=[\rho_{n,1}] [ \ell_{2n}] [\rho_{n,1}]^{-1}[ \rho_{n,2}][\ell_{2n+1}][\rho_{n,2}]^{-1}\in \pi_1(\mathbb{P},b_0).
\end{equation}

\begin{lemma}\label{dn-inj}
  For every $k\geqslant 2n+1$, $d_{n\#}:\pi_1(\mathbb{H}_{n,k}^+,b_0)\rightarrow \pi_1(\mathbb{H}_{n+1,k}^+,b_0)$ is injective.
\end{lemma}

\begin{proof}
For $m\in \{n,n+1\}$, the group $\pi_1(\mathbb{H}_{m,k}^+,b_0)$ is free on the set \[\{[\rho_{i,j}]\mid 1\leqslant i \leqslant m-1, 1\leqslant j \leqslant 2\}\cup\{[\ell_i]\mid m\leqslant i \leqslant k\}.\]
So, the claim  follows from Equation~(\ref{peel}) and the fact that $d_n\circ \rho_{i,j}=\rho_{i,j}$ for  $1\leqslant i \leqslant n-1$, $1\leqslant j \leqslant 2$ and $d_n\circ \ell_i=\ell_i$ for $n+1\leqslant i\leqslant k$.
\end{proof}

\begin{Notation}
 We will denote functions $\sigma:\mathcal{L}\rightarrow \mathcal{N}$ to inverse limits
\begin{eqnarray*} \displaystyle \mathcal{N} &=& \lim_{\longleftarrow} \left(\mathcal{N}_1\stackrel{\kappa_1}{\longleftarrow} \mathcal{N}_2\stackrel{\kappa_2}{\longleftarrow} \mathcal{N}_3\stackrel{\kappa_3}{\longleftarrow}\cdots\right) \\ &=&\left\{(x_n)_{n\geqslant 1}\in \prod_{n=1}^\infty \mathcal{N}_n\mid \kappa_n(x_{n+1})=x_n \mbox{ for all }n\geqslant 1\right\}
\end{eqnarray*}
 as sequences $\sigma=(\sigma_n)_{n\geqslant 1}$ with $\sigma_n=\mu_n\circ \sigma:{\mathcal L}\rightarrow {\mathcal N}_n$, where $\mu_n:\mathcal{N}\rightarrow \mathcal{N}_n$ are the projections.
\end{Notation}

\begin{lemma}\label{dn-com-inj}
For every $n$,  $d_{n\#}:\pi_1(\mathbb{H}_{n}^+,b_0)\rightarrow \pi_1(\mathbb{H}_{n+1}^+,b_0)$ is injective.
\end{lemma}

\begin{proof}
  Let $1\not=[\alpha]\in \pi_1(\mathbb{H}_n^+,b_0)$. Since \[ \mathbb{H}^+_n=\lim_{\longleftarrow} \left( \mathbb{H}^+_{n,n-1} \stackrel{r_{n,n-1}}{\longleftarrow} \mathbb{H}^+_{n,n} \stackrel{r_{n,n}}{\longleftarrow} \mathbb{H}^+_{n,n+1} \stackrel{r_{n,n+1}}{\longleftarrow} \cdots \right)\approx \mathbb{H},\]
  we have that \[((r_{n,k})_\#)_{k\geqslant n-1}:\pi_1(\mathbb{H}_n^+,b_0)\rightarrow \lim_{\longleftarrow}\left( \pi_1(\mathbb{H}_{n,n-1}^+,b_0) \stackrel{(r_{n,n-1})_\#}{\longleftarrow}  \pi_1(\mathbb{H}_{n,n}^+,b_0) \stackrel{(r_{n,n})_\#}{\longleftarrow}
     \cdots\right)\] is injective \cite{CC2000}.
Therefore,  there is a $k\geqslant 2n+1$ such that $(r_{n,k})_\#([\alpha])\not=1$.
  The claim now follows from Lemma~\ref{dn-inj} and applying $\pi_1$ to the diagram above.
\end{proof}

\begin{lemma}\label{direct} $\pi_1(\mathbb{P},b_0)$ is isomorphic to the direct limit \[\lim_{\longrightarrow} \left(  \pi_1(\mathbb{H}_{1}^+,b_0)\stackrel{d_{1\#}}{\longrightarrow} \pi_1(\mathbb{H}_{2}^+,b_0) \stackrel{d_{2\#}}{\longrightarrow} \pi_1(\mathbb{H}_{3}^+,b_0) \stackrel{d_{3\#}}{\longrightarrow} \cdots \right)\]
with canonical homomorphisms $\iota_{n\#}:\pi_1(\mathbb{H}_n^+,b_0)\rightarrow \pi_1(\mathbb{P},b_0)$ induced by inclusion $\iota_n:\mathbb{H}_n^+\hookrightarrow \mathbb{P}$.
\end{lemma}

\begin{proof} For $n\in \mathbb{N}$ and $[\alpha]\in \pi_1(\mathbb{H}_n^+,b_0)$, we have $[d_n\circ \alpha]=[\alpha]\in \pi_1(\mathbb{P},b_0)$. Hence $\iota_{n+1\#}\circ d_{n\#}=\iota_{n\#}$. In order to verify the  universal property, let $h_n: \pi_1(\mathbb{H}_n^+,b_0)\rightarrow G$ be  a homomorphisms with $h_{n+1}\circ d_{n\#}=h_n$. Let $[\alpha]\in \pi_1(\mathbb{P},b_0)$. Since $\alpha([0,1])$ is compact, there is an $n\in \mathbb{N}$ such that $\alpha([0,1])\cap P^\circ_i=\emptyset$ for all $i\geqslant n$. Put $\beta=d_{n-1}\circ d_{n-2} \circ \cdots \circ d_1\circ\alpha$. Then $\beta:([0,1],\{0,1\})\rightarrow (\mathbb{H}_n^+,b_0)$ and $\iota_{n\#}([\beta])=[\alpha]\in \pi_1(\mathbb{P},b_0)$. Moreover,  $h_{n+1}([d_n\circ \beta])=h_n([\beta])$. Put $h([\alpha])=h_n([\beta])$. Once we show that $h:\pi_1(\mathbb{P},b_0)\rightarrow G$ is well-defined, it is clear that $h$ is a homomorphism with $h\circ \iota_{n\#}=h_n$ and that $h$ is the unique homomorphism with this property. To this end, suppose that $[\alpha]=[\widetilde{\alpha}]\in \pi_1(\mathbb{P},b_0)$.
Let $H:[0,1]\times [0,1]\rightarrow \mathbb{P}$ be a homotopy with $H(t,0)=\alpha(t)$, $H(t,1)=\widetilde{\alpha}(t)$, and $H(0,t)=H(1,t)=b_0$ for all $t\in [0,1]$. Since $H([0,1]\times [0,1])$ is compact, for $n\in\mathbb{N}$ sufficiently large, we have  $d_{n-1}\circ \cdots \circ d_2\circ d_1\circ H([0,1]\times [0,1])\subseteq \mathbb{H}_{n}^+$ so that $[d_{n-1}\circ \cdots \circ  d_2\circ d_1\circ \alpha]=[d_{n-1}\circ \cdots \circ  d_2\circ d_1\circ \widetilde{\alpha}]\in \pi_1(\mathbb{H}_{n}^+,b_0)$.
\end{proof}

\begin{lemma}\label{HP-inj}
For every $n$, $\iota_{n\#}:\pi_1(\mathbb{H}_n^+,b_0)\rightarrow \pi_1(\mathbb{P},b_0)$ is injective.
\end{lemma}

\begin{proof} This follows from Lemmas~\ref{dn-com-inj} and \ref{direct}.
\end{proof}

Recall that a group $G$ is called {\em locally free} if every finitely generated subgroup of $G$ is free.

\begin{proposition}\label{lf}
$\pi_1(\mathbb{P},b_0)$ is locally free.
\end{proposition}

\begin{proof}
Since $\pi_1(\mathbb{H}^+_n,b_0)$ is isomorphic to a subgroup of an inverse limit of free groups of finite rank, it is locally free \cite[Theorem~1]{CF1959b}.
Then, by Lemma~\ref{direct}, $\pi_1(\mathbb{P},b_0)$ is a direct limit of locally free groups, and thus locally free \cite[Lemma~24]{CHM}.
\end{proof}

\begin{remark}[Metric Hawaiian Pants $\mathbb{P}^\ast \subseteq \mathbb{R}^3$]
While $\mathbb{P}$ is not metrizable (it is not first countable at $b_0$), it is naturally homotopy equivalent to the metrizable space formed by attaching  the pants $P_n$ to the union $\bigcup_{i=1}^\infty Z_i\subseteq \mathbb{R}^3$ of the cylinders $Z_n=C_n\times[1-n,n-1]$ via identifying $\partial D_n$ with $C_n \times \{1-n\}$, $\partial D_{n,1}$ with $C_{2n} \times \{2n-1\}$, and $\partial D_{n,2}$ with $C_{2n+1} \times \{2n\}$.

If we change the embedding $\mathbb{H}\subseteq \mathbb{R}^2$, this procedure yields a subspace of $\mathbb{R}^3$ that is homotopy equivalent to $\mathbb{P}$: For each $n\in \mathbb{N}$, choose a triangle $C'_n\subseteq \mathbb{R}^2$ of diameter less than $1/n$ such that for all $i\not=j$, $C'_i\cap C'_j=\{b_0\}$ and the bounded components of $\mathbb{R}^2\setminus C'_i$ and $\mathbb{R}^2\setminus C'_j$ are disjoint. Then $\bigcup_{i=1}^\infty C'_i \subseteq \mathbb{R}^2$ is homeomorphic to $\mathbb{H}$. The adjunction space resulting from attaching the pants $P_n$ to the union $\bigcup_{i=1}^\infty Z_i'\subseteq \mathbb{R}^3$ of the corresponding cylinders $Z_n'$  can now readily be implemented in $\mathbb{R}^3$ by forming the union of $\bigcup_{i=1}^\infty Z_i'$ with appropriate sets $P_n'\subseteq \mathbb{R}^3$ that are homeomorphic to $P_n$.

To obtain a subspace $\mathbb{P}^\ast\subseteq\mathbb{R}^3$ which is semilocally simply connected at all but one point and homotopy equivalent to $\mathbb{P}$, slightly deform the cylinders $Z_n'$  so that their only common point of contact is the origin.
Then there is a bijective homotopy equivalence $h:\mathbb{P}\rightarrow \mathbb{P}^\ast$ such that $h(C_n)=C_n'\times\{0\}$ for all $n\in \mathbb{N}$ with homotopy inverse $g:\mathbb{P}^\ast\rightarrow \mathbb{P}$ collapsing the cylinders such that $g|_{P'_n}: P_n'\rightarrow P_n$ is a homeomorphism, $g(Z_n')= C_n$, $g|_{C_n'\times\{0\}}=(h|_{C_n})^{-1}$, and $g(h(B_{n,j}))=B_{n,j}$ for all $n\in \mathbb{N}$ and $j\in \{1,2\}$.
The resulting deformation retraction of the cylinders allows the proofs in this paper for those properties of $\mathbb{P}$ that are not homotopy invariant (such as the 1-UV$_0$ property and the discrete monodromy property) to go through for $\mathbb{P}^\ast$ with only minor changes. However,  working with $\mathbb{P}$ is conceptually simpler.
\end{remark}

\section{Non-injectivity into the first \v{C}ech homotopy group}\label{sec:notinj}

Let $X$ be a path-connected topological space and $x_0\in X$. For a point $x\in X$, let ${\mathcal T}_x$ denote the set of all open neighborhoods of $x$ in $X$. For $U\in {\mathcal T}_x$ and a path $\alpha$ in $X$ from $x_0$ to $x$, consider the subgroup \[\pi(\alpha,U)=\{[\alpha\cdot \delta\cdot \alpha^-]\mid \delta \mbox{ a loop in } U\}\leqslant \pi_1(X,x_0).\]

Let $\pi(x,U)$ denote the normal closure of $\pi(\alpha,U)$ in $\pi_1(X,x_0)$, i.e., the subgroup  generated by all  $\pi(\beta,U)$, with $\beta$ a path from $x_0$ to $x$:
\[\pi(x,U)=\left<\pi(\alpha,U)\mid \alpha(1)=x\right> \leqslant \pi_1(X,x_0).\]

\begin{remark}\label{holes} Suppose $U\in {\mathcal T}_x$ is path connected and let $[\gamma]\in \pi_1(X,x_0)$. Then $[\gamma]\in \pi(x,U)$ if and only if there is a map $g:D\setminus(D^\circ_1\cup D^\circ_2\cup \cdots\cup D^\circ_j)\rightarrow X$ from a ``disk with holes'' to $X$ with $g|_{\partial D}=\gamma$ and $g(\partial D_i)\subseteq U$  for all $i\in\{1,2,\dots,j\}$. (Here, $D_1, D_2, \dots, D_j$ are pairwise disjoint disks in the interior $D^\circ$ of the disk $D$.)
\end{remark}

Let $Cov(X)$ denote the set of all open covers of $X$. For ${\mathcal U}\in Cov(X)$, let  $\pi({\mathcal U},x_0)$ denote the subgroup of $\pi_1(X,x_0)$ generated by all $\pi(x,U)$ with $U\in {\mathcal U}$ and $x\in U$.

\begin{definition}[Spanier group \cite{Spanier}]
The subgroup \[\pi^s(X,x_0)=\bigcap_{{\mathcal U}\in Cov(X)} \pi({\mathcal U},x_0)=\bigcap_{{\mathcal U}\in Cov(X)} \left<\pi(x,U)\mid x\in U\in {\mathcal U}\right>\]
of the fundamental group $\pi_1(X,x_0)$ is called the {\em Spanier group} of $X$.
\end{definition}

\begin{remark}\label{classical} If $X$ is locally path connected, then  for a given  $H\leqslant \pi_1(X,x_0)$, there is a (classical) covering projection $p:(\widetilde{X},\widetilde{x})\rightarrow (X,x_0)$ with $p_\#\pi_1(\widetilde{X},\widetilde{x})=H$ if and only if there is a ${\mathcal U}\in Cov(X)$ such that $\pi({\mathcal U},x_0)\leqslant H$ \cite{Spanier}.
\end{remark}

\begin{remark}\label{kernel}
The Spanier group $\pi^s(X,x_0)$ is contained in the kernel of the natural homomorphism $\Psi_X:\pi_1(X,x_0)\rightarrow \check{\pi}_1(X,x_0)$ to the first \v{C}ech homotopy group \cite{FZ2007}. If $X$ is locally path connected and metrizable, then $\pi^s(X,x_0)$ equals this kernel \cite{BFa2}.
\end{remark}

Let $\iota=\iota_1:\mathbb{H}\hookrightarrow \mathbb{P}$ denote inclusion.

\begin{lemma}\label{intoS}  $\iota_\#\pi_1(\mathbb{H},b_0)\leqslant \pi^s(\mathbb{P},b_0)$.
\end{lemma}

\begin{proof}
 Let $[\alpha]\in \pi_1(\mathbb{H},b_0)$ and  $\mathcal{U}\in Cov(\mathbb{P})$. Choose $U\in {\mathcal U}$ with $b_0\in U$ and fix $n\in \mathbb{N}$ with $\iota(\mathbb{H}_n)\subseteq U$.   Express $[\alpha]=[\gamma_1][\delta_1][\gamma_2][\delta_2]\cdots[\gamma_m][\delta_m]$ with loops $\gamma_j$ in $C_1\cup C_2\cup \cdots \cup C_{n-1}$ and loops  $\delta_j$ in $\mathbb{H}_n$. Then $\iota_\#([\delta_j])\in \pi(b_0,U)$ for every $j$. Also, for every $j$, we have $\iota_\#([\gamma_j])\in \left<[\ell_1],[\ell_2],\dots,[\ell_{n-1}]\right>\leqslant \pi_1(\mathbb{P},b_0)$, so that repeated application of Equation~(\ref{relation}) yields $\iota_\#([\gamma_j])\in \pi(b_0,U)$. Hence, $\iota_\#([\alpha])\in \pi^s(\mathbb{P},b_0)$.
\end{proof}

\begin{theorem}\label{thm:notinj}
$\Psi_\mathbb{P}:\pi_1(\mathbb{P},b_0)\rightarrow \check{\pi}_1(\mathbb{P},b_0)$ is not injective.
\end{theorem}

\begin{proof}

    Combining Lemma~\ref{HP-inj} (with $n$=1), Lemma~\ref{intoS}, and  Remark~\ref{kernel}, we obtain $1\not=\iota_\#\pi_1(\mathbb{H},b_0)\leqslant \pi^s(\mathbb{P},b_0)\leqslant \ker \Psi_\mathbb{P}$.
\end{proof}

Recall that if the fundamental group of a Peano continuum does not (canonically) inject into the first \v{C}ech homotopy group, then it is not residually n-slender \cite{EF}. However, since $\mathbb{P}$ is not a Peano continuum, we verify this separately:

\begin{definition}[Noncommutatively slender \cite{E1992}]
A group $G$ is called {\em noncommutatively slender} ({\em n-slender} for short) if for every homomorphism $h:\pi_1(\mathbb{H},b_0)\rightarrow G$, there is
a $k\in \mathbb{N}$ such that $h([\alpha])=1$ for all loops $\alpha:([0,1],\{0,1\})\rightarrow (\mathbb{H}_k,b_0)$.
\end{definition}

 Recall that a group $G$ is called {\em residually n-slender} (respectively {\em residually free}) if for every $1\not=g\in G$ there is an n-slender (respectively free) group $S$ and a homomorphism $h:G\rightarrow S$, such that $h(g)\not=1$.

\begin{proposition}\label{notres}
$\pi_1(\mathbb{P},b_0)$ is not residually n-slender.
\end{proposition}

\begin{proof}
Consider $1\not=[\ell_1]\in \pi_1(\mathbb{P},b_0)$. Let $h:\pi_1(\mathbb{P},b_0) \rightarrow S$ be a homomorphism to an n-slender group $S$. It suffices to show that $h([\ell_1])=1$. If we precompose $h$ with the homomorphism  $\iota_\#:\pi_1(\mathbb{H},b_0)\rightarrow \pi_1(\mathbb{P},b_0)$,  induced by inclusion $\iota:\mathbb{H}\hookrightarrow \mathbb{P}$, and note that $S$ is n-slender, we see that $h([\ell_k])=h\circ \iota_\#([\ell_k])=1$ for all but finitely many $k$. However, by Equation~(\ref{relation}), we have $[\ell_n]=[\rho_{n,1}][\ell_{2n}][\rho_{n,1}]^{-1} [ \rho_{n,2}][ \ell_{2n+1}][ \rho_{n,2}]^{-1}$ in $\pi_1(\mathbb{P},b_0)$ for all $n$. Hence, $h([\ell_n])=1$ for all $n$.
\end{proof}

 \begin{remark}\label{NRF}
 Every free group is n-slender \cite{E1992}. So, $\pi_1(\mathbb{P},b_0)$ is not residually free.
\end{remark}

\section{The strongly homotopically Hausdorff property}\label{sec:SHH}

\begin{definition}[Strongly homotopically Hausdorff \cite{CMRZZ}] \label{SHH} A path-connected space $X$ is called {\em strongly homotopically Hausdorff at} $x\in X$ if for every essential loop $\gamma$ in $X$, there is an open neighborhood $U$ of $x$ in $X$ such that $\gamma$ cannot be freely homotoped into $U$, that is, if
\[\bigcap_{U\in {\mathcal T}_x} \bigcup_{\alpha(1)\in U} \pi(\alpha,U)=\{1\}.\]
(If $U$ is path connected, we may replace ``$\alpha(1)\in U$'' by ``$\alpha(1)=x$''.)
The space $X$ is called {\em strongly homotopically Hausdorff} if it is strongly homotopically Hausdorff at every point $x\in X$.
\end{definition}

\begin{remark}\label{Z'}
If the natural homomorphism  $\pi_1(X,x) \hookrightarrow \check{\pi}_1(X,x)$ is injective, then $X$ is {\em strongly homotopically Hausdorff} \cite{FRVZ}. However,  the converse does not hold in general \cite[Example $Z'$]{FRVZ}.
\end{remark}

\begin{remark}[The Hawaiian Mapping Torus]
Let $f:\mathbb{H}\rightarrow \mathbb{H}$ be the map given by $f\circ\ell_{n}=\ell_{n+1}$, $n\in\mathbb{N}$. The \textit{Hawaiian Mapping Torus} is the space $M_f=\mathbb{H}\times [0,1]/\sim$, where $(x,0)\sim (f(x),1)$ for all $x\in\mathbb{H}$. Identifying $\mathbb{H}$ with the image of $\mathbb{H}\times \{0\}$ in $M_f$, the inclusion $i:\mathbb{H}\hookrightarrow M_f$ induces an injection $i_\#:\pi_1(\mathbb{H},b_0)\rightarrow \pi_1(M_f,b_0)$. Consider the loop $\rho:([0,1],\{0,1\})\rightarrow (M_f,b_0)$ where $\rho(s)$ is the image of $(b_0,s)$ in $M_f$ and put $t=[\rho]\in \pi_1(M_f,b_0)$. From two applications of van Kampen's Theorem, we get that $\pi_1(M_f,b_0)$ is isomorphic to the quotient of $\pi_1(\mathbb{H},b_0)\ast \left< \,t\, \right>$ by the relations  $g=t f_{\#}(g)t^{-1}$, $g\in\pi_1(\mathbb{H},b_0)$ (see \cite{SW}). Iterating these relations, we see that each $[\gamma]\in \pi_1(\mathbb{H},b_0)$ factors in $\pi_1(M_f,b_0)$, for every $n\in\mathbb{N}$, as a conjugate $[\rho]^n[f^n\circ\gamma][\rho]^{-n}$ where the diameter of the loop $f^n\circ\gamma$ shrinks to $0$.

While $M_f$ is a Peano continuum that embeds into $\mathbb{R}^3$ and has many of the same properties as $\mathbb{P}$, it is not strongly homotopically Hausdorff, since $\ell_1$ is freely homotopic to $\ell_n$ for all $n\in\mathbb{N}$. Our detailed treatment of $\mathbb{P}$ is motivated by the fact that $\pi_1(\mathbb{P},b_0)$ exhibits a somewhat more intricate algebraic phenomenon: in order to write an element $g\in\iota_\#\pi_1(\mathbb{H},b_0)\leqslant \pi_1(\mathbb{P},b_0)$ as a product of conjugates of homotopy classes of arbitrarily small loops (as in the proof of Lemma~\ref{intoS}), it takes an exponentially growing number of distinct conjugating elements, namely products of $[\rho_{i,j}]$.
\end{remark}

A path $\alpha:[a,b]\rightarrow X$ is called {\em reduced} if for every $a\leqslant s<t\leqslant b$ with $\alpha(s)=\alpha(t)$, the loop $\alpha|_{[s,t]}$ is not null-homotopic in $X$. For a one-dimensional metric space $X$, every path $\alpha:[a,b]\rightarrow X$ is homotopic (relative to endpoints) within $\alpha([a,b])$ to either a constant path or a reduced path, which is unique up to reparametrization \cite{E2002}. A path $\alpha:[a,b]\rightarrow X$ is called {\em cyclically reduced} if $\alpha\cdot \alpha$ is reduced.

\begin{lemma}[Lemma~3.11 of \cite{BFi}]\label{cancellinglemma}
Let $\lambda:([0,1],0)\to (X,x_0)$ be a reduced path in a one-dimensional metric space $X$, $\delta:[0,1]\to X$ be a reduced loop based at $\lambda(1)$, and $\gamma$ be a reduced representative of $[\lambda\cdot \delta\cdot \lambda^{-}]$. Then there exist $s,t\in [0,1]$ such that
$\lambda|_{[0,t]}\circ \phi=\gamma|_{[0,s]}$, for some increasing homeomorphism $\phi:[0,s]\rightarrow [0,t]$,
and $\lambda([t,1])\subseteq \delta([0,1])$.
\end{lemma}

\begin{theorem}\label{SHHthm}
$\mathbb{P}$ is strongly homotopically Hausdorff.
\end{theorem}

\begin{proof}For $n\in \mathbb{N}$, we define  $U_n=\mathbb{H}\cap\{(x,y)\in \mathbb{R}^2\mid x<\frac{2n+1}{n(n+1)}\}$. Since $\operatorname{diam}(C_{n+1})=\frac{2}{n+1}<\frac{2n+1}{n(n+1)}<\frac{2}{n}=\operatorname{diam}(C_n)$, we see that the sequence $U_1\supseteq U_2\supseteq U_3 \supseteq \cdots$ forms a  neighborhood basis for $\mathbb{H}$ at $b_0$. For every pair $n,k\in \mathbb{N}$, $f_k^{-1}(U_n)$ has three components $L^0_{k,n}$, $L^1_{k,n}$, and $L^2_{k,n}$, each of which is an open arc or a circle (there are four cases based on the position of $n$ relative to $k$, $2k$, and $2k+1$), such that $a_k\in L^0_{k,n}\subseteq \partial D_k$, $b_k\in L^1_{k,n}\subseteq \partial D_{k,1}$ and $c_k\in L^2_{k,n}\subseteq \partial D_{k,2}$. Since, for a given $k$, $L^i_{k,n+1}\subseteq L^i_{k,n}$, we may choose three pairwise disjoint open neigborhoods $N^0_{k,n}$, $N^1_{k,n}$, $N^2_{k,n}$ of $L^0_{k,n}$, $L^1_{k,n}$, $L^2_{k,n}$ in $P_k$, respectively, such that $N^i_{k,n+1}\subseteq N^i_{k,n}$,  $N^i_{k,n}\cap \partial P_k=L^i_{k,n}$, and $N^i_{k,n}$  deformation retracts onto $L^i_{k,n}$.  Define $V_{k,n}= N^0_{k,n}\cup N^1_{k,n}\cup N^2_{k,n}$.
 (See Figure~\ref{collar}.) Put $V_n=\bigcup_{k\in\mathbb{N}} f_k(V_{k,n})$. Then $V_n$ is an open neighborhood of $b_0$ in $\mathbb{P}$, $V_{n+1}\subseteq V_n$, and $V_n$ deformation retracts onto $U_n$. (Note that $V_1\supseteq V_2\supseteq V_3\supseteq \cdots$ does not form a neighborhood basis for $\mathbb{P}$ at $b_0$.)

\begin{figure}
  \centering
  \includegraphics[scale=1]{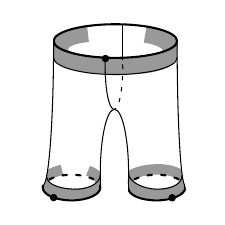}
  \caption{$V_{k,n}$ (gray region) with $k<2k\leqslant n<2k+1$}\label{collar}
\end{figure}

Now suppose, to the contrary, that $\mathbb{P}$ is not strongly homotopically Hausdorff. Then $\mathbb{P}$ is not strongly homotopically Hausdorff at $b_0$, since $\mathbb{P}\setminus\{b_0\}$ is locally contractible. Hence, there is a $1\not=[\gamma]\in\pi_1(\mathbb{P},b_0)$ such that for every $n\in \mathbb{N}$, there is a path $\lambda_n:[0,1]\rightarrow \mathbb{P}$ from $\lambda_n(0)=b_0$ to $\lambda_n(1)\in V_n$ with $[\gamma]\in \pi(\lambda_n,V_n)$.

Since each $V_n$ is path connected, we may assume that $\lambda_n(1)=b_0$. Choose loops $\delta_n$ in $V_n$ with $[\gamma]=[\lambda_n][\delta_n][\lambda_n]^{-1}\in \pi_1(\mathbb{P},b_0)$.  Since $V_n$ deformation retracts onto $U_n$, we may assume that $\delta_n$ lies in $U_n$. We may also assume that each $\delta_n$ is {\em reduced} in $\mathbb{H}$.
 This implies that $\delta_n$ lies in $\mathbb{H}_{n+1}$, because $C_m$ is not fully contained in $U_n$ if $m\geqslant n$.
Replacing each $\lambda_n$ by  $\lambda_1^-\cdot \lambda_n$, we may assume that $\gamma=\delta_1$, which lies in $U_1$.
There is a maximal $s_0\in [0,1]$ such that $\gamma|_{[0,s_0]}$ is a reparametrization of $(\gamma|_{[t_0,1]})^-$ for some $t_0\in (s_0,1]$. Then $\gamma(s_0)=\gamma(t_0)=b_0$, $\gamma|_{[s_0,t_0]}$ is cyclically reduced, and $[\gamma]=[\gamma|_{[0,s_0]}][\gamma|_{[s_0,t_0]}][\gamma|_{[0,s_0]}]^{-1}$. Replacing each $\lambda_n$ by $\gamma|_{[0,s_0]}^-\cdot \lambda_n$ and replacing $\gamma$ by $\gamma|_{[s_0,t_0]}$, we may assume that $\gamma$ is {\em cyclically} reduced.

Let $m$ be minimal such that $\gamma([0,1])$ intersects $C_m\setminus \{b_0\}$. Then $\gamma$ fully traverses $C_m$, at least once,
and $\gamma$ lies in $\mathbb{H}_m$. Let $F$ be a homotopy from $\gamma$ to $\lambda_{m}\cdot \delta_{m}\cdot \lambda_{m}^-$ (relative to endpoints) within $\mathbb{P}$. Choose $n>m$ such that the image of $F$ misses all $P^\circ_i$ with $i>n$. Then $F'=d_n\circ d_{n-1}\circ \cdots \circ d_1\circ F$ is a homotopy from $\gamma'=d_n\circ d_{n-1}\circ \cdots \circ d_1\circ \gamma$ to $\lambda\cdot\delta\cdot \lambda^-$ (relative to endpoints) within $\mathbb{H}^+_{n+1}$, where $\lambda=d_n\circ d_{n-1}\circ \cdots \circ d_1\circ \lambda_{m}$ and $\delta=d_n\circ d_{n-1}\circ \cdots \circ d_1\circ \delta_{m}$.

On one hand, $\gamma'$ is a cyclically reduced loop in $\mathbb{H}^+_{n+1}$ which traverses $B_{m,1}$. On the other hand,
the image of $\delta$  is disjoint from $B_{m,1}\setminus\{b_0\}$. (In particular, $\lambda$ is not null-homotopic in $\mathbb{H}^+_{n+1}$.)
Therefore, using reduced representatives $\lambda'$ and $\delta'$ of $[\lambda]$ and $[\delta]$, respectively, we have $s,t>0$ in Lemma~\ref{cancellinglemma}. Applying Lemma~\ref{cancellinglemma} to $\gamma'^-$, as well, we see that $\gamma'$ is not cyclically reduced; a contradiction.
\end{proof}

\section{An inverse limit of finitely generated free monoids}\label{wordsequences}

 Let ${\mathcal W}^+_{n,k}$ denote the set of finite words (including the empty word) over the alphabet ${\mathcal A}_{n,k}=\{\rho^{\pm 1}_{i,j}\mid 1\leqslant i \leqslant n-1, 1\leqslant j\leqslant 2\}\cup\{\ell_i^{\pm 1}\mid n\leqslant i \leqslant k\}$.
 Then ${\mathcal W}^+_{n,k}$ forms a free monoid on the set ${\mathcal A}_{n,k}$ under concatenation.

The deletion of a subword of the form  $\rho_{i,j}^{+1}\rho_{i,j}^{-1}$, $\rho_{i,j}^{-1}\rho_{i,j}^{+1}$,  $\ell_i^{+1}\ell_i^{-1}$, or $\ell_i^{-1}\ell_i^{+1}$ from an element of ${\mathcal W}^+_{n,k}$ is called a {\em cancellation}. A word is called {\em reduced} if it does not allow for any cancellation. Recall that, starting with a fixed element of ${\mathcal W}^+_{n,k}$, every maximal sequence of cancellations results in the same reduced word.

  Let $F^+_{n,k}\subseteq {\mathcal W}^+_{n,k}$ be the subset of all reduced words. Then $F^+_{n,k}$ forms a free group on the set  $\{\rho_{i,j}\mid 1\leqslant i \leqslant n-1,1\leqslant j\leqslant 2\}\cup\{\ell_i\mid n\leqslant i \leqslant k\}$  under concatenation, followed by maximal cancellation. Moreover, we have  isomorphisms  $h_{n,k}:\pi_1(\mathbb{H}^+_{n,k},b_0)\rightarrow  F^+_{n,k}$, mapping $[\rho_{i,j}]\mapsto \rho_{i,j}$ and $[\ell_i]\mapsto \ell_i$.

Let $R_{n,k}: {\mathcal W}^+_{n,k+1}\rightarrow {\mathcal W}^+_{n,k}$ denote the function that deletes every occurrence of the letters $\ell_{k+1}$ and $\ell_{k+1}^{-1}$ from a word. Let $S_{n,k}:{\mathcal W}^+_{n,k}\rightarrow F^+_{n,k}$ be the function that maximally cancels words and put $T_{n,k}=S_{n,k}\circ R_{n,k}|_{F^+_{n,k+1}}:F^+_{n,k+1}\rightarrow F^+_{n,k}$.
Then $R_{n,k}$, $S_{n,k}$, and $T_{n,k}$ are monoid/group homomorphisms.

Define $g_{n,k}=h_{n,k}\circ r_{n,k\#}: \pi_1(\mathbb{H}^+_n,b_0)\rightarrow F^+_{n,k}\subseteq  {\mathcal W}^+_{n,k}$.
Since $T_{n,k}\circ g_{n,k+1}=g_{n,k}:\pi_1(\mathbb{H}^+_n,b_0)\rightarrow F^+_{n,k}$ and
$T_{n,k}\circ h_{n,k+1}=h_{n,k}\circ r_{n,k\#}:\pi_1(\mathbb{H}^+_{n,k+1},b_0)\rightarrow F^+_{n,k}$, we have, for each $n$,
 as in the proof of Lemma~\ref{dn-com-inj}, an injective homomorphism into an inverse limit $F^+_n$ of free groups:
\[g_n=(g_{n,k})_{k\geqslant n-1}:\pi_1(\mathbb{H}^+_n,b_0)\rightarrow F^+_n=\lim_{\longleftarrow} \left( F^+_{n,n-1}\stackrel{T_{n,n-1}}{\longleftarrow} F^+_{n,n}\stackrel{T_{n,n}}{\longleftarrow}F^+_{n,n+1}\stackrel{T_{n,n+1}}{\longleftarrow}\cdots\right).\]
(Note that $g_n$ is not surjective and that $F^+_n$ is not free.)

\begin{remark}\label{LEC} The image of $g_n$ equals the locally eventually constant sequences, where $(y_{n,k})_{k\geqslant n-1}\in F^+_n$ is called  {\em locally eventually constant} if for every $m\geqslant n-1$, the sequence $(R_{n,m}\circ R_{n,m+1} \circ \cdots \circ R_{n,k-1}(y_{n,k}))_{k\geqslant m}$ is eventually constant \cite{MM}.
\end{remark}

For every $n\in\mathbb{N}$, $x\in \pi_1(\mathbb{H}^+_n,b_0)$, and $m\geqslant n-1$, the sequence $(R_{n,m}\circ R_{n,m+1} \circ \cdots \circ R_{n,k-1}(g_{n,k}(x)))_{k\geqslant m}$ of (unreduced) words is eventually constant, so that we may define functions
$\omega_{n,m}:\pi_1(\mathbb{H}^+_n,b_0)\rightarrow \mathcal{W}^+_{n,m}$ by \[\omega_{n,m}(x)=R_{n,m}\circ R_{n,m+1} \circ \cdots \circ R_{n,k-1}(g_{n,k}(x)),\] for sufficiently large $k$, and obtain commutative diagrams
\[
\xymatrix{
& {\mathcal W}^+_n \ar[d]^{S_n}   & \hspace{-30pt} =  \displaystyle  \lim_{\longleftarrow} ( {\mathcal W}^+_{n,n-1} \ar[d]^{S_{n,n-1}} & \ar[l]_<(0.2){R_{n,n-1}} \ar[d]^{S_{n,n}} {\mathcal W}^+_{n,n} & \ar[l]_<(0.2){R_{n,n}} \ar[d]^{S_{n,n+1}} {\mathcal W}^+_{n,n+1} & \ar[l]_<(0.2){R_{n,n+1}} \cdots) \\
\pi_1(\mathbb{H}^+_n,b_0)  \ar@{^{(}->}[ru]^{\omega_n} \ar@{^{(}->}[r]^<(0.3){g_n} & F^+_n & \hspace{-30pt} =  \displaystyle  \lim_{\longleftarrow} ( F^+_{n,n-1} & \ar[l]_<(0.2){T_{n,n-1}} F^+_{n,n} & \ar[l]_<(0.2){T_{n,n}} F^+_{n,n+1} & \ar[l]_<(0.2){T_{n,n+1}} \cdots) }
\]
with injective functions $\omega_n=(\omega_{n,k})_{k\geqslant n-1}$ that output {\em stabilized} word sequences. (See \cite{DTW} or \cite{FZ2013b}, for example.)

For $k\geqslant 2n+1$, let $D_{n}: {\mathcal W}^+_{n,k} \rightarrow {\mathcal W}^+_{n+1,k}$ be the monomorphism that replaces every occurrence of the letter $\ell_n$ by $\rho_{n,1}\ell_{2n} \rho_{n,1}^{-1} \rho_{n,2} \ell_{2n+1} \rho_{n,2}^{-1}$ and every occurrence of the letter
$\ell_n^{-1}$ by $\rho_{n,2} \ell_{2n+1}^{-1} \rho_{n,2}^{-1}\rho_{n,1}\ell_{2n}^{-1} \rho_{n,1}^{-1}$. Then, for every $k\geqslant 2n+1$,  we obtain the following commutative diagram:
\vspace{-15pt}

\[
\xymatrix@=8pt{
&&\pi_1(\mathbb{H}^+_{n,k+1},b_0)\ar[rrr]^{d_{n\#}} \ar[ddd]^{r_{n,k\#}} \ar[dl]_<(0.3){h_{n,k+1}} &&&\pi_1(\mathbb{H}^+_{n+1,k+1},b_0) \ar[ddd]^{r_{n+1,k\#}} \ar[dl]_<(0.3){h_{n+1,k+1}}\\
&F^+_{n,k+1}\ar[rrr]^{D_n} \ar[ddd]^{T_{n,k}}&&&F^+_{n+1,k+1}\ar[ddd]^{T_{n+1,k}} &\\
{\mathcal W}^+_{n,k+1}\ar[rrr]^{D_n}\ar[ddd]^{R_{n,k}} \ar[ru]^<(0.0){S_{n,k+1}}&&&{\mathcal W}^+_{n+1,k+1}\ar[ddd]^{R_{n+1,k}} \ar[ru]^<(0.0){S_{n+1,k+1}}&&\\
&&\pi_1(\mathbb{H}^+_{n,k},b_0)\ar[rrr]^{d_{n\#}} \ar[dl]_<(0.3){h_{n,k}}&&&\pi_1(\mathbb{H}^+_{n+1,k},b_0) \ar[dl]^{h_{n+1,k}}\\
&F^+_{n,k}\ar[rrr]^{D_n} &&&F^+_{n+1,k}&\\
{\mathcal W}^+_{n,k}\ar[rrr]^{D_n} \ar[ru]^{S_{n,k}}&&&{\mathcal W}^+_{n+1,k} \ar[ru]_{S_{n+1,k}}&&
}
\]
\vspace{-5pt}

\noindent Note: For every $\omega\in F^+_{n,k}$, the word $D_n(\omega)$ is reduced.

\begin{remark} In view of the above, the direct limit structure of Lemma~\ref{direct} suggests the possibility of labelling the elements of $\pi_1(\mathbb{P}, b_0)$ using sequences of finite words over finite alphabets that gradually exclude all letters $\ell_n^{\pm 1}$ in favor of only using the conjugating letters $\rho_{n,j}^{\pm 1}$. Since such a shift causes conjugating pairs to become adjacent in words at later levels, we use monoid structures to prevent their cancellation. This, in turn, requires us to work with $D_{n}: {\mathcal W}^+_{n,k} \rightarrow {\mathcal W}^+_{n+1,k}$ in what follows, at a level just deep enough to stabilize the appropriate word sequences.
\end{remark}

Recall that ${\mathcal W}^+_{n+1,n}$ is the set of finite words over $\{\rho_{i,j}^{\pm 1}\mid 1\leqslant i \leqslant n, 1\leqslant j \leqslant 2\}$. In particular, the set ${\mathcal W}^+_{1,0}$ contains only one element: the empty word.

Let $E_{n-1}:{\mathcal W}^+_{n+1,n}\rightarrow {\mathcal W}^+_{n,n-1}$ denote the epimorphism deleting every occurrence of the letters $\rho^{\pm 1}_{n,1}$ and $\rho^{\pm 1}_{n,2}$. Then, for $k\geqslant 2n+1$, we obtain commutative trapezoids:

\begin{equation}\label{trapezoid}
\xymatrix{
{\mathcal W}^+_{1,k} \ar[r]^{D_1} \ar[d]_{R_{1,k-1}}& {\mathcal W}^+_{2,k} \ar[r]^{D_2} \ar[d]_{R_{2,k-1}}& \cdots \ar[r]^{D_{n-1}}& {\mathcal W}^+_{n,k} \ar[r]^{D_n} \ar[d]_{R_{n,k-1}}& {\mathcal W}^+_{n+1,k} \ar[d]_{R_{n+1,k-1}}\\
\vdots \ar[d]_{R_{1,n}} & \vdots \ar[d]_{R_{2,n}}&  & \vdots\ar[d]_{R_{n,n}} & \vdots\ar[d]_{R_{n+1,n}}  \\
{\mathcal W}^+_{1,n} \ar[d]_{R_{1,n-1}}& {\mathcal W}^+_{2,n} \ar[d]_{R_{2,n-1}}&  & {\mathcal W}^+_{n,n} \ar[d]_{R_{n,n-1}}& {\mathcal W}^+_{n+1,n} \ar[dl]^{E_{n-1}}\\
 &  & & {\mathcal W}^+_{n,n-1} \ar[dl]^{E_{n-2}} \\
\vdots\ar[d]_{R_{1,1}} &\vdots \ar[d]_{R_{2,1}}& \iddots \ar[dl]^{E_1}\\
{\mathcal W}^+_{1,1}   \ar[d]_{R_{1,0}}& {\mathcal W}^+_{2,1} \ar[dl]^{E_0}  \\
{\mathcal W}^+_{1,0}
}
\end{equation}

\begin{theorem}\label{inj}
There is a well-defined injective  function \[\chi=(\chi_n)_{n\in \mathbb{N}}:\pi_1(\mathbb{P},b_0)\hookrightarrow \lim_{\longleftarrow}\left( {\mathcal W}^+_{2,1}\stackrel{E_1}{\longleftarrow} {\mathcal W}^+_{3,2}\stackrel{E_2}{\longleftarrow} {\mathcal W}^+_{4,3}\stackrel{E_3}{\longleftarrow} \cdots \right)\] defined as follows: for a given $[\alpha]\in \pi_1(\mathbb{P},b_0)$ and sufficiently large $n\geqslant 2$, choose $[\beta]\in\pi_1(\mathbb{H}_n^+,b_0)$ with $\iota_{n\#}([\beta])=[\alpha]$ and  put $\chi_{n-1}([\alpha])=\omega_{n,n-1}([\beta])\in {\mathcal W}^+_{n,n-1}$.
\end{theorem}

\begin{proof} First we show that $\chi$ is well-defined.  By Lemmas~\ref{direct} and \ref{HP-inj} it suffices to show that for any  $[\beta]\in\pi_1(\mathbb{H}_n^+,b_0)$, we have $E_{n-1}(\omega_{n+1,n}(d_{n\#}([\beta])))=\omega_{n,n-1}([\beta])$, making the following diagram commute:
\[\xymatrix{
\mathcal{W}_{2,1}^{+}
  &
\mathcal{W}_{3,2}^{+} \ar[l]_-{E_1}  &
\cdots \ar[l]_-{E_2} &
\mathcal{W}_{n,n-1}^{+} \ar[l]_-{E_{n-2}} &
\mathcal{W}_{n+1,n}^{+} \ar[l]_-{E_{n-1}} &
\cdots \ar[l]  \\
  \pi_1(\mathbb{H}_{2}^{+},b_0) \ar@{^{(}->}[r]_-{d_{2\#}} \ar[u]_-{\omega_{2,1}} &
\pi_1(\mathbb{H}_{3}^{+},b_0) \ar@{^{(}->}[r]_-{d_{3\#}} \ar[u]_-{\omega_{3,2}} &
\cdots \ar@{^{(}->}[r]_-{d_{n-1\#}} &
\pi_1(\mathbb{H}_{n}^{+},b_0) \ar@{^{(}->}[r]_-{d_{n\#}} \ar[u]_-{\omega_{n,n-1}}  &  \pi_1(\mathbb{H}_{n+1}^{+},b_0) \ar[u]_-{\omega_{n+1,n}}  \ar@{^{(}->}[r]
& \cdots \\
}\]
(Recall that the underlying set of a direct limit of groups is the direct limit of the underlying sets.) To this end, put $\beta'=d_{n}\circ \beta$. Then $[\beta']\in\pi_1(\mathbb{H}_{n+1}^+,b_0)$.
By definition of $\omega_{n,n-1}$ and $\omega_{n+1,n}$, respectively, for sufficiently large $k\geqslant 2n+1$, we have (cf.\@ Remark~\ref{LEC})
\[\omega_{n,n-1}([\beta])= R_{n,n-1}\circ R_{n,n}\circ \cdots \circ R_{n,k-1}\circ g_{n,k}([\beta])\]
and
\[\omega_{n+1,n}([\beta'])= R_{n+1,n}\circ R_{n+1,n+1} \circ \cdots \circ R_{n+1,k-1}\circ g_{n+1,k}([\beta']).\]
Noting that $g_{n+1,k}([\beta'])=
h_{n+1,k}\circ r_{n+1,k\#}([\beta'])=h_{n+1,k}\circ r_{n+1,k\#}\circ d_{n\#}([\beta])
=h_{n+1,k}\circ d_{n\#}\circ r_{n,k\#}([\beta])=D_{n}\circ h_{n,k}\circ r_{n,k\#}([\beta])
=D_n\circ g_{n,k}([\beta])$
 and that \[E_{n-1}\circ R_{n+1,n}\circ R_{n+1,n+1} \circ \cdots \circ R_{n+1,k-1}\circ D_n=R_{n,n-1}\circ R_{n,n}\circ \cdots \circ R_{n,k-1}\]
 (see Diagram~(\ref{trapezoid})) we obtain the desired equality:
\begin{eqnarray*}
E_{n-1}(\omega_{n+1,n}([\beta']))&=&E_{n-1}\circ R_{n+1,n}\circ R_{n+1,n+1} \circ \cdots \circ R_{n+1,k-1}\circ g_{n+1,k}([\beta'])\\
&=&R_{n,n-1}\circ R_{n,n}\circ \cdots \circ R_{n,k-1}\circ g_{n,k}([\beta])\\
&=&\omega_{n,n-1}([\beta]).
\end{eqnarray*}

Now we show that $\chi$ is injective. Suppose  $[\alpha^{(1)}]\not=[\alpha^{(2)}]\in \pi_1(\mathbb{P},b_0)$. Choose $n\geqslant 2$ sufficiently large (as in the proof of Lemma~\ref{direct}) so that $\beta^{(s)}([0,1])\subseteq \mathbb{H}_{n}^+$, where $\beta^{(s)}=d_{n-1}\circ d_{n-2} \circ \cdots \circ d_1\circ\alpha^{(s)}$ for $s\in\{1,2\}$.
  Then $\iota_{n\#}([\beta^{(s)}])=[\alpha^{(s)}]$ and $[\beta^{(1)}]\not=[\beta^{(2)}]\in \pi_1(\mathbb{H}^+_n,b_0)$. Hence, there is an $m\geqslant n$ such that $\omega_{n,m-1}([\beta^{(1)}])\not=\omega_{n,m-1}([\beta^{(2)}])$. We may assume that $m$ is even. Put $\gamma^{(s)}=d_{m-1}\circ d_{m-2}\circ \cdots \circ d_n\circ \beta^{(s)}$. Choose $k\geqslant 2(m-1)+1$ sufficiently large, such that for $s\in \{1,2\}$, we have:
\[\xymatrix{ g_{n,k}([\beta^{(s)}]) \ar@{|->}_{R_{n,m-1}\circ \cdots \circ R_{n,k-1}}[d] \ar@{|->}^{D_{m-1}\circ D_{m-2}\circ  \cdots \circ D_n}[rr] & & g_{m,k}([\gamma^{(s)}]) \ar@{|->}^{R_{m,m-1}\circ \cdots \circ R_{m,k-1}}[d] \\
\omega_{n,m-1}([\beta^{(s)}])  & & \hspace{.8in} \omega_{m,m-1}([\gamma^{(s)}])=\chi_{m-1}([\alpha^{(s)}])
}\]

For $n\leqslant j\leqslant m-1$, let $\overline{D}_{j,m-1}:\mathcal{W}_{j,m-1}^{+}\rightarrow  \mathcal{W}_{j+1,m-1}^{+}$ be the monomorphism that  replaces every occurrence of the letter $\ell_{j}$ (respectively $\ell_j^{-1}$)  by $\rho_{j,1}\ell_{2j}\rho_{j,1}^{-1}\rho_{j,2}\ell_{2j+1}\rho_{j,2}^{-1}$ (respectively $\rho_{j,2}\ell_{2j+1}^{-1}\rho_{j,2}^{-1}\rho_{j,1}\ell_{2j}^{-1}\rho_{j,1}^{-1}$) if $2j+1\leqslant m-1$, but instead replaces it by  $\rho_{j,1}\rho_{j,1}^{-1}\rho_{j,2}\rho_{j,2}^{-1}$ (respectively $\rho_{j,2}\rho_{j,2}^{-1}\rho_{j,1}\rho_{j,1}^{-1}$) if $2j\geqslant m$.

     Since each $\overline{D}_{j,m-1}$ with $n\leqslant j\leqslant m-1$ is  injective, so is their composition  $\overline{D}=\overline{D}_{m-1,m-1}\circ \overline{D}_{m-2,m-1}\circ\cdots \circ \overline{D}_{n,m-1}$. Moreover, the following diagram commutes:
\[\xymatrix{
\mathcal{F}_{n,k}^{+} \ar[d]_{R_{n,m-1}\circ \cdots \circ R_{n,k-1}} \ar[rrrr]^-{D_{m-1}\circ D_{m-2}\circ \cdots\circ D_n} &&&&\mathcal{F}_{m,k}^{+} \ar[d]^{R_{m,m-1}\circ \cdots \circ R_{m,k-1}}\\
\mathcal{W}_{n,m-1}^{+} \ar[rrrr]_-{\overline{D}} &&&& \mathcal{W}_{m,m-1}^{+}
}\]
Hence, for $s\in\{1,2\}$, we have
\begin{eqnarray*}
\overline{D}(\omega_{n,m-1}([\beta^{(s)}])) &=&
\overline{D}(R_{n,m-1}\circ \cdots \circ R_{n,k-1}(g_{n,k}([\beta^{(s)}])))\\
 &=&
R_{m,m-1}\circ \cdots \circ R_{m,k-1}\circ D_{m-1}\circ D_{m-2}\circ \cdots\circ D_n(g_{n,k}([\beta^{(s)}]))\\
&=& R_{m,m-1}\circ \cdots \circ R_{m,k-1}(g_{m,k}([\gamma^{(s)}]))\\
&=& \omega_{m,m-1}([\gamma^{(s)}])\\
&=& \chi_{m-1}([\alpha^{(s)}]).
\end{eqnarray*}
Since $\omega_{n,m-1}([\beta^{(1)}])\neq \omega_{n,m-1}([\beta^{(2)}])$ and $\overline{D}$ is injective, we have $\chi_{m-1}([\alpha^{(1)}])\neq \chi_{m-1}([\alpha^{(2)}])$.
\end{proof}

\begin{remark}\label{noalgebra}
Although $\bigcup_{i=1}^\infty (B_{i,1}\cup B_{i,2})\subseteq \mathbb{P}$ is a bouquet of circles that is not homeomorphic to $\mathbb{H}$, one can algebraically set up a commutative diagram as follows:
\[
\xymatrix{
{\mathcal W}^+_{2,1} \ar[d]_{S_{2,1}}& {\mathcal W}^+_{3,2} \ar[l]_{E_1} \ar[d]_{S_{3,2}} & {\mathcal W}^+_{4,3} \ar[l]_{E_2} \ar[d]_{S_{4,3}} & \cdots \ar[l]_{E_3}\\
F^+_{2,1}  & F^+_{3,2} \ar[l]_{J_1} & F^+_{4,3} \ar[l]_{J_2} & \cdots \ar[l]_{J_3}
}
\]
However, there does \underline{not} exist an injective homomorphism \[\pi_1(\mathbb{P},b_0)\hookrightarrow \lim_{\longleftarrow}\left( F^+_{2,1}\stackrel{J_1}{\longleftarrow} F^+_{3,2}\stackrel{J_2}{\longleftarrow} F^+_{4,3}\stackrel{J_3}{\longleftarrow} \cdots \right),\]
because $\pi_1(\mathbb{P},b_0)$ is not residually free (Remark~\ref{NRF}).
\end{remark}

\section{The homotopically path Hausdorff property}\label{sec:HPH}

\begin{definition}[Homotopically path Hausdorff \cite{FRVZ}] \label{DefHPH} A path-connected space $X$ is called {\em homotopically path Hausdorff} if for every two paths $\alpha, \beta:[0,1] \rightarrow X$ with $\alpha(0)=\beta(0)$ and $\alpha(1)=\beta(1)$ such that $\alpha\cdot \beta^-$ is not null-homotopic,
there is a partition $0=t_0<t_1<\cdots < t_n=1$ of $[0,1]$ and open subsets $U_1, U_2, \dots, U_n$ of $X$ with $\alpha([t_{i-1},t_i])\subseteq U_i$ for all $1\leqslant i\leqslant n$ and with the property that if $\gamma:[0,1]\rightarrow X$ is any path with $\gamma([t_{i-1},t_i])\subseteq U_i$ for all $1\leqslant i\leqslant n$ and with $\gamma(t_i)=\alpha(t_i)$ for all $0\leqslant i \leqslant n$, then $\gamma\cdot \beta^-$ is not null-homotopic.
\end{definition}

\begin{remark}\label{T1} We recall from \cite{BFa} that a connected and locally path-connected space $X$ is homotopically path Hausdorff if and only if $\pi_1(X,x)$ is T$_1$ in the quotient topology induced by the compact-open topology on the loop space $\Omega(X,x)$.
\end{remark}

\begin{remark} \label{Y'}
If the natural homomorphism  $\pi_1(X,x) \hookrightarrow \check{\pi}_1(X,x)$ is injective, then $X$ is homotopically path Hausdorff \cite{FRVZ}. However, the converse does not hold in general \cite[Example $Y'$]{FRVZ}.
\end{remark}

\begin{theorem}\label{HPH}
$\mathbb{P}$ is homotopically path Hausdorff.
\end{theorem}

\begin{proof} Let $1\not=[\alpha]\in \pi_1(\mathbb{P},b_0)$. We wish to find a partition $0=t_0<t_1<\cdots<t_s=1$ and open subsets $U_1, U_2, \dots, U_s\subseteq \mathbb{P}$ with $\alpha([t_{i-1},t_i])\subseteq U_i$ for all $1\leqslant i \leqslant s$ such that the following property holds: if $\gamma:[0,1]\rightarrow \mathbb{P}$ is any loop with $\gamma(t_i)=\alpha(t_i)$ for all $0\leqslant i\leqslant s$ and $\gamma([t_{i-1},t_i])\subseteq U_i$ for all $1\leqslant i \leqslant s$, then $[\gamma]\not=1\in\pi_1(\mathbb{P},b_0)$. (This property is preserved if we add a subdivision point $t'\in (t_{i-1},t_i)$ and choose $U'=U_i$. Therefore, checking this statement for all essential loops $\alpha$  based at $b_0$, validates Definition~\ref{DefHPH} for all essential loops $\alpha\cdot \beta^-$.)

By Theorem~\ref{inj}, there is an $n\in \mathbb{N}$ such that for $\beta=d_{n-1}\circ d_{n-2}\circ \cdots \circ d_1 \circ \alpha$, we have $\beta([0,1])\subseteq \mathbb{H}^+_n$ and $\chi_{n-1}([\alpha])=\omega_{n,n-1}([\beta])\in {\mathcal W}^+_{n,n-1}$ is not the empty word. Choose $k\in \mathbb{N}$ sufficiently large, so that
\[R_{n,n-1}\circ R_{n,n}\circ \cdots \circ R_{n,k-1}\circ g_{n,k}([\beta])=\omega_{n,n-1}([\beta]).\]

 Consider $\beta:[0,1]\rightarrow \mathbb{H}^+_n\subseteq \mathbb{P}^+_n$ and $r_{n,k}:\mathbb{H}^+_n\rightarrow \mathbb{H}^+_{n,k}$.
Since $\mathbb{H}^+_{n,k}$ is a finite bouquet of circles, there is an open cover $\mathcal B$ of $\mathbb{H}^+_{n,k}$
 with the following property: if $\eta,\tau:([0,1],\{0,1\})\rightarrow (\mathbb{H}^+_{n,k},b_0)$ are two loops that are $\mathcal B$-close, i.e., if for every $t\in [0,1]$, there is a $B\in {\mathcal B}$ with  $\{\eta(t),\tau(t)\}\subseteq B$, then $[\eta]=[\tau]\in \pi_1(\mathbb{H}^+_{n,k},b_0)$. Choose an open cover $\{W_j \mid j\in J\}$ of $\mathbb{H}^+_n$ and a cover $\{V_j\mid j\in J\}$ of $\mathbb{H}^+_n$ by open subsets of $\mathbb{P}^+_n$ such that $\{W_j \mid j\in J\}$ refines $\{(r_{n,k})^{-1}(B)\mid B\in {\mathcal B}\}$ and such that each $V_j$ deformation retracts onto $W_j$. (See proof of Theorem~\ref{SHHthm}.) Choose a partition $0=t_0<t_1<\cdots<t_s=1$ and indices $j_1, j_2, \dots , j_s\in J$ such that $\beta([t_{i-1},t_i])\subseteq V_{j_i}$ for $1\leqslant i \leqslant s$. Put  $U_i=(d_{n-1}\circ d_{n-2}\circ \cdots \circ d_1)^{-1}(V_{j_i})$ for $1\leqslant i \leqslant s$.

Now, let $\gamma:[0,1]\rightarrow \mathbb{P}$ be a loop with $\gamma(t_i)=\alpha(t_i)$ for all $0\leqslant i\leqslant s$ and $\gamma([t_{i-1},t_i])\subseteq U_i$ for all $1\leqslant i \leqslant s$. Then $\gamma:[0,1]\rightarrow \mathbb{P}$ is homotopic (relative to its
endpoints) to a loop $\gamma':[0,1]\rightarrow \mathbb{H}^+_n\subseteq \mathbb{P}$ such that $r_{n,k}\circ \gamma':[0,1]\rightarrow \mathbb{H}^+_{n,k}$ and $r_{n,k}\circ \beta:[0,1]\rightarrow \mathbb{H}^+_{n,k}$ are $\mathcal B$\nobreakdash-close. Hence, $[r_{n,k}\circ \gamma']=[r_{n,k}\circ \beta]\in \pi_1(\mathbb{H}^+_{n,k},b_0)$, so that $g_{n,k}([\gamma'])=h_{n,k}([r_{n,k}\circ \gamma'])=h_{n,k}([r_{n,k}\circ \beta])=g_{n,k}([\beta])$.
Choose $m\geqslant k$ sufficiently large, so that
\[R_{n,n-1}\circ R_{n,n}\circ \cdots \circ R_{n,k-1}\circ R_{n,k}\circ \cdots \circ R_{n,m-1}\circ g_{n,m}([\gamma'])=\omega_{n,n-1}([\gamma']),\] which, as a word in ${\mathcal W}^+_{n,n-1}$,  has at least as many letters as
\begin{eqnarray*} & & R_{n,n-1}\circ R_{n,n}\circ \cdots \circ R_{n,k-1}\circ T_{n,k}\circ \cdots \circ T_{n,m-1}\circ g_{n,m}([\gamma'])\\&=&R_{n,n-1}\circ R_{n,n}\circ \cdots \circ R_{n,k-1}\circ g_{n,k}([\gamma'])\\&=&R_{n,n-1}\circ R_{n,n}\circ \cdots \circ R_{n,k-1}\circ g_{n,k}([\beta])\\&=&\omega_{n,n-1}([\beta]).
\end{eqnarray*}
Hence, $\chi_{n-1}([\gamma'])=\omega_{n,n-1}([\gamma'])$ is not the empty word. We conclude that $\chi([\gamma])=\chi([\gamma'])$ is not trivial so that $[\gamma]\not=1\in\pi_1(\mathbb{P},b_0)$.
\end{proof}

\section{The 1-UV$_0$ property}

 \begin{definition}[1-UV$_0$ \cite{CMRZZ}]\label{small}
   We say that $X$ is {\em 1-UV$_0$ at} $x\in X$ if for every  neighborhood $U$ of $x$ in $X$, there is an open subset $V$ in $X$ with $x\in V\subseteq U$  such that for every map $f:D^2\rightarrow X$ from the unit disk with $f(S^1)\subseteq V$ there is a map $g:D^2\rightarrow U$ with $f|_{S^1}=g|_{S^1}$.

   We say that $X$ is {\em 1-UV$_0$} if $X$ is 1-UV$_0$ at every point $x\in X$.
 \end{definition}

\begin{proposition}\label{puv}
All one-dimensional spaces and all planar spaces are 1-UV$_0$.
\end{proposition}

\begin{proof}
In a one-dimensional space, every null-homotopic loop contracts within its own image \cite[Lemma~2.2]{CF1959}. In a planar space, every null-homotopic loop has a contraction whose diameter equals that of the image of the loop \cite[Lemma~13]{FZ2005}.
\end{proof}

\begin{theorem}\label{UVThm}
  $\mathbb{P}$ is 1-UV$_0$.
\end{theorem}

\begin{proof}
It suffices to show that $\mathbb{P}$ is 1-UV$_0$ at $b_0$. Let $U$ be an open neighborhood of $b_0$ in $\mathbb{P}$.

Recall  $U_n$ and $f^{-1}_k(U_n)=L^0_{k,n}\cup L^1_{k,n} \cup L^2_{k,n}\subseteq \partial P_k$ from the proof of Theorem~\ref{SHHthm}. Fix $m\in \mathbb{N}$ such that $b_0\in U_m\subseteq U\cap \mathbb{H}$. For each $k\in \mathbb{N}$, choose three pairwise disjoint open neighborhoods $M^0_k,M^1_k,M^2_k$ of $L^0_{k,m},L^1_{k,m},L^2_{k,m}$ in $P_k$, respectively,  such that, for $i\in \{0,1,2\}$,  $M^i_k\cap \partial P_k=L^i_{k,m}$,  $M^i_k\cap P^\circ_k\subseteq U\cap P^\circ_k$, and $M^i_k$  deformation retracts onto $L^i_{k,m}$.
Define $V'_{k,m}= M^0_k\cup M^1_k\cup M^2_k$ and $V=\bigcup_{k\in\mathbb{N}} f_k(V'_{k,m})$. Then $V$ is an open subset of $\mathbb{P}$ with $b_0\in V\subseteq U$.

Let $f:D^2\rightarrow \mathbb{P}$ be a map with $f(S^1)\subseteq V$.
We will show that $\alpha=f|_{S^1}$ contracts within $V$. Since $V$ is path connected, we may assume that $\alpha$ is a loop based at $b_0$ and  show that $[\alpha]=1\in \pi_1(V,b_0)$. Since $V$ deformation retracts onto $U_m\subseteq\mathbb{H}$, we may assume that $\alpha$ lies in $\mathbb{H}$. Since $[\alpha]=1\in \pi_1(\mathbb{P},b_0)$, we have $[\alpha]=1\in \pi_1(\mathbb{H},b_0)$ by Lemma~\ref{HP-inj}. As $\mathbb{H}$ is one-dimensional, this implies that $\alpha$ contracts within its own image.
\end{proof}

\section{Generalized covering projections and the homotopically Hausdorff property}\label{review}

In this section, we briefly review  generalized covering projections. Let $X$ be a path-connected topological space and $H \leqslant \pi_1(X,x_0)$. Even if there is no classical covering projection $p:(\widetilde{X},\widetilde{x})\rightarrow (X,x_0)$ with $p_\#\pi_1(\widetilde{X},\widetilde{x})=H$ (see Remark~\ref{classical}), there might be a {\em generalized} one:

\begin{definition}[Generalized covering projection \cite{B,FZ2007}]
Let $X$ be a path-connected topological space. We call a map $q:\widehat{X}\rightarrow X$ a {\em generalized covering projection} if $\widehat{X}$ is nonempty, connected and locally path connected and if for every $\widehat{x}\in \widehat{X}$, for every connected and locally path-connected space $Y$, and for every map $f:(Y,y)\rightarrow (X,q(\widehat{x}))$ with $f_\#\pi_1(Y,y)\leqslant  q_\#\pi_1(\widehat{X},\widehat{x})$, there is a unique map $g:(Y,y)\rightarrow (\widehat{X},\widehat{x})$ such that $q\circ g=f$.
\end{definition}

\begin{remark} Suppose $q:(\widehat{X},\widehat{x})\rightarrow (X,x_0)$ is a generalized covering projection. Then $q_\#:\pi_1(\widehat{X},\widehat{x})\rightarrow \pi_1(X,x_0)$ is injective. If we put $K=q_\#\pi_1(\widehat{X},\widehat{x})$, then $q:\widehat{X}\rightarrow X$ is characterized as usual, up to equivalence, by the conjugacy class of $K$ in $G=\pi_1(X,x_0)$. Moreover, $Aut(\widehat{X}\stackrel{q}{\rightarrow} X)\cong N_G(K)/K$, where $N_G(K)$ denotes the normalizer of $K$ in $G$. (The standard arguments apply \cite[2.3.5 \& 2.6.2]{Spanier}.)
\end{remark}

If it exists, a generalized covering projection can be obtained in the standard way:
   On the set of all paths $\alpha:([0,1],0)\rightarrow (X,x_0)$ consider the equivalence relation $\alpha\sim \beta$ if and only if $\alpha(1)=\beta(1)$ and $[\alpha\cdot \beta^-]\in H$. Denote the equivalence class of $\alpha$ by $\left<\alpha\right>$ and denote the set of all equivalence classes by $\widetilde{X}_H$. Let $\widetilde{x}_0$ denote the class containing the constant path at $x_0$. We give $\widetilde{X}_H$  the topology generated by basis elements of the form $\left<\alpha,U\right>=\{\left<\alpha\cdot \gamma \right> \mid \gamma:([0,1],0)\rightarrow (U,\alpha(1))\}$, where $U$ is an open subset of $X$ and $\left<\alpha\right>\in \widetilde{X}_H$ with $\alpha(1)\in U$. Then $\widetilde{X}_H$ is connected and locally path connected and the endpoint projection $p_H:\widetilde{X}_H\rightarrow X$, defined by $p_H(\left<\alpha\right>)=\alpha(1)$, is a continuous surjection. Moreover, the map $p_H:\widetilde{X}_H\rightarrow X$ is  open  if and only if $X$ is locally path connected.

If $p_H:\widetilde{X}_H\rightarrow X$ has unique path lifting, then it is a generalized covering projection and, for every $\left<\alpha\right>\in \widetilde{X}_H$,
$(p_H)_\#:\pi_1(\widetilde{X}_H,\left<\alpha\right>)\rightarrow \pi_1(X,\alpha(1))$ is a monomorphism onto $[\alpha^-]H[\alpha]$; in particular $(p_H)_\#\pi_1(\widetilde{X}_H,\widetilde{x}_0)=H$ \cite{FZ2007}.

If $X$ admits a generalized covering projection $q:(\widehat{X},\widehat{x})\rightarrow (X,x_0)$ such that $q_\#\pi_1(\widehat{X},\widehat{x})=H$, then there is a homeomorphism $h:(\widehat{X},\widehat{x})\rightarrow (\widetilde{X}_H,\widetilde{x}_0)$ with $p_H\circ h=q$ \cite{B}.

\begin{remark}\label{HH} For  $p_H:\widetilde{X}_H\rightarrow X$ to have unique path lifting, every fiber $p_H^{-1}(x)$ with $x\in X$ must be $T_1$, but not necessarily discrete \cite{FZ2007}. Moreover, T$_1$ fibers are not sufficient \cite{BFi,VZ}. Note that these fibers are T$_1$ if and only if for every $x\in X$, \[\bigcup_{\alpha(1)=x}\;\; \bigcap_{U\in {\mathcal T}_x} H\pi(\alpha,U)=H.\]
\end{remark}

\begin{definition}[Homotopically Hausdorff rel.\@ $H$ \cite{FZ2007}]
We call $X$ {\em homotopically Hausdorff relative to} $H\leqslant \pi_1(X,x_0)$, if every fiber of $p_H:\widetilde{X}_H\rightarrow X$ is $T_1$. We call $X$ {\em homotopically Hausdorff} if it is homotopically Hausdorff relative to $H=\{1\}$.
\end{definition}

\begin{remark}
If $X$ is strongly homotopically Hausdorff, then $X$ is homotopically Hausdorff. (Compare the formula of Remark~\ref{HH} with that in Definition~\ref{SHH}.)
\end{remark}

\begin{remark}\label{NormalNotHH}
There are normal subgroups $H\trianglelefteqslant \pi_1(\mathbb{H},b_0)$, such that $\mathbb{H}$ is not homotopically Hausdorff relative to $H$. For example, $\mathbb{H}$ is not homotopically Hausdorff relative to the commutator subgroup of $\pi_1(\mathbb{H},b_0)$ \cite[Example~3.10]{BFi}.
\end{remark}

We abbreviate $\widetilde{X}_{\{1\}}$ by $\widetilde{X}$ and $p_{\{1\}}:\widetilde{X}_{\{1\}}\rightarrow X$ by $p:\widetilde{X}\rightarrow X$.
 Moreover, note that if $H=\{1\}$, then $\left<\alpha\right>=[\alpha]$.

\begin{remark}\label{HPH->UPL}
If $X$ is homotopically path Hausdorff, then  $p:\widetilde{X}\rightarrow X$ has unique path lifting \cite{FRVZ}.
\end{remark}

\begin{remark}\label{1UV->UPL} If $X$ is path connected, 1-UV$_0$ and metrizable, then $p:\widetilde{X}\rightarrow X$ has unique path lifting \cite{BFi}.
\end{remark}

\begin{theorem}\label{Ucov} There exists a generalized covering projection $p:\widetilde{\mathbb{P}}\rightarrow \mathbb{P}$ with $\pi_1(\widetilde{\mathbb{P}},\widetilde{b}_0)=\{1\}$.
\end{theorem}

\begin{proof}
By Theorem~\ref{HPH} and Remark~\ref{HPH->UPL}, $p:\widetilde{\mathbb{P}}\rightarrow \mathbb{P}$ has unique path lifting. Hence, $p:\widetilde{\mathbb{P}}\rightarrow \mathbb{P}$ is a generalized covering projection, $p_\#\pi_1(\widetilde{\mathbb{P}},\widetilde{b}_0)=\{1\}$, and $p_\#:\pi_1(\widetilde{\mathbb{P}},\widetilde{b}_0)\rightarrow \pi_1(\mathbb{P},b_0)$ is injective.
\end{proof}
\
\section{The discrete monodromy property}\label{sec:DMP}

Let $X$ be a path-connected topological space and $H \leqslant \pi_1(X,x_0)$.
Even if the map $p_H:\widetilde{X}_H\rightarrow X$ does not have unique path lifting, for every path $\beta:[0,1]\rightarrow X$ and every $\left<\alpha\right>\in p_H^{-1}(\beta(0))$ there is a continuous {\em standard path lift} $\widetilde{\beta}:([0,1],0)\rightarrow (\widetilde{X}_H,\left<\alpha\right>)$ with $p_H\circ \widetilde{\beta}=\beta$, defined by $\widetilde{\beta}(t)=\left<\alpha\cdot \beta_t\right>$, where $\beta_t(s)=\beta(ts)$.

Based on the standard path lift, we may define the {\em standard monodromy} for $p_H:\widetilde{X}_H\rightarrow X$, as follows. For a path $\beta:[0,1]\rightarrow X$ from $\beta(0)=x$ to $\beta(1)=y$, we define $\Phi_\beta:p_H^{-1}(x)\rightarrow p_H^{-1}(y)$ by $\Phi_\beta(\left<\alpha\right>)=\left<\alpha\cdot \beta\right>$.

Clearly, $\Phi_\beta:p_H^{-1}(x)\rightarrow p_H^{-1}(y)$ is a bijective function with inverse $\Phi_\beta^{-1}=\Phi_{\beta^-}$. However, $\Phi_\beta$ need not be continuous. (Such is the case for $(X,x_0)=(\mathbb{H},b_0)$ and $H=\{1\}$, although $p:\widetilde{X}\rightarrow X$ has unique path lifting. See \cite{FGa} for a discussion.)

\begin{remark}
Note that $\Phi_\beta$ depends only on the homotopy class $[\beta]$. Moreover, $\left<\alpha\right>\ast\left<\beta\right>:=\Phi_\beta(\left<\alpha\right>)$ is a well-defined group operation on $p_H^{-1}(x_0)$  if and only if $H$ is a normal subgroup of $\pi_1(X,x_0)$.
\end{remark}

\begin{definition}[Discrete monodromy]\label{DefDM}
We say that $X$ has the {\em discrete mono\-dromy property relative to} $H\leqslant \pi_1(X,x_0)$ if for every $x, y\in X$ and for every path\linebreak $\beta:[0,1]\rightarrow X$ from $\beta(0)=x$ to $\beta(1)=y$, the monodromy  $\Phi_\beta:p_H^{-1}(x)\rightarrow p_H^{-1}(y)$ is either the identity function or its graph is a discrete subset of $p_H^{-1}(x)\times p_H^{-1}(y)\subseteq \widetilde{X}_H\times \widetilde{X}_H$. We say that $X$ has the {\em discrete monodromy property} if it has the discrete monodromy property relative to $H=\{1\}$.
\end{definition}

\begin{remark}
Clearly, if every fiber of $p_H:\widetilde{X}_H\rightarrow X$ is discrete, then $X$ has the discrete monodromy property relative to $H$. However, the converse does not hold in general: $p:\widetilde{\mathbb{H}}\rightarrow \mathbb{H}$ has the discrete monodromy property (see Proposition~\ref{1dm}), but $p^{-1}(b_0)$ is not discrete.
\end{remark}

\begin{remark}\hspace{10pt}\label{ID}

\begin{itemize} \item[(a)]  $\Phi_\beta:p_H^{-1}(x)\rightarrow p_H^{-1}(y)$ is the identity function if and only if $x=y$ and $[\beta]\in [\alpha^-]H[\alpha]$ for all paths $\alpha$ in $X$ from $x_0$ to $x$.

\item[(b)] The graph of the identity function $id_{p_H^{-1}(x)} : p_H^{-1}(x)\rightarrow p_H^{-1}(x)$ is discrete if and only if $p_H^{-1}(x)$ is discrete.
\item[(c)] $p_H^{-1}(x)$ is discrete if and only if for every path $\alpha$ in $X$ from $x_0$ to $x$, there is a $U\in {\mathcal T}_x$ such that $\pi(\alpha,U)\subseteq H$, i.e.,  $H\pi(\alpha,U)=H$.
\end{itemize}
\end{remark}

\begin{lemma}\label{discreteEQ} Let $H\leqslant \pi_1(X,x_0)$.
The graph of $\Phi_\beta:p_H^{-1}(x)\rightarrow p_H^{-1}(y)$ is discrete if and only if for every path $\alpha$ in $X$ from $x_0$ to $x$, there are $U\in {\mathcal T}_x$  and $V\in {\mathcal T}_y$ such that $H\pi(\alpha,U)\cap H\pi(\alpha\cdot \beta,V)=H$.
\end{lemma}

\begin{proof} First, observe that if $f:A\rightarrow B$ is an injective function between topological spaces, then its graph $\Gamma=\{(a,b)\in A\times B\mid f(a)=b\}$ is a discrete subset of $A\times B$ if and only if for every $a\in A$  there are $U\in {\mathcal T}_a$ and $V\in {\mathcal T}_{f(a)}$ such that $f(U)\cap V=\{f(a)\}$.

Now, suppose that the graph of $\Phi_\beta:p_H^{-1}(x)\rightarrow p_H^{-1}(y)$ is discrete and let $\left<\alpha\right>\in p_H^{-1}(x)$. Choose $U\in {\mathcal T}_x$  and $V\in {\mathcal T}_y$ with \[\Phi_\beta(\left<\alpha,U\right>\cap p_H^{-1}(x))\cap \left(\left<\alpha\cdot \beta,V\right>\cap p_H^{-1}(y)\right)=\{\left<\alpha\cdot \beta\right>\}.\] Let $g\in H\pi(\alpha,U)\cap H\pi(\alpha\cdot \beta,V)$. Then $g=h_1[\alpha\cdot \gamma\cdot \alpha^-]$ for some $h_1\in H$ and some loop $\gamma$ in $U$, and $g=h_2[\alpha\cdot \beta \cdot \delta \cdot \beta^-\cdot \alpha^-]$ for some $h_2\in H$ and some loop $\delta$ in $V$. Hence, $\Phi_\beta(\left<\alpha\cdot \gamma\right>)=\left<\alpha\cdot \gamma\cdot \beta\right>=\left<\alpha\cdot \beta\cdot  \delta\right>\in \left<\alpha\cdot \beta,V\right>\cap p_H^{-1}(y)$. Also, $\Phi_\beta(\left<\alpha\cdot \gamma\right>)\in \Phi_\beta(\left<\alpha,U\right>\cap p_H^{-1}(x))$. Therefore, $\left<\alpha\cdot\gamma\cdot \beta\right>= \Phi_\beta(\left<\alpha\cdot \gamma\right>)=\left<\alpha\cdot \beta\right>$, so that $g=h_1[\alpha\cdot \gamma\cdot \alpha^-]\in H$.

Conversely, let $\left<\alpha\right>\in p_H^{-1}(x)$ and suppose $U\in {\mathcal T}_x$  and $V\in {\mathcal T}_y$ are such that $H\pi(\alpha,U)\cap H\pi(\alpha\cdot \beta,V)= H$. Let \[\widetilde{x}\in \Phi_\beta(\left<\alpha,U\right>\cap p_H^{-1}(x))\cap \left(\left<\alpha\cdot \beta,V\right>\cap p_H^{-1}(y)\right).\] Then $\widetilde{x}=\left<\alpha\cdot \gamma\cdot \beta\right>$ for some loop $\gamma$ in $U$, and $\widetilde{x}=\left<\alpha\cdot \beta\cdot \delta\right>$ for some loop $\delta$ in $V$. Then $[\alpha\cdot \gamma\cdot \alpha^-][\alpha\cdot \beta\cdot \delta^-\cdot \beta^-\cdot \alpha^-]\in H$, so that $[\alpha\cdot \gamma\cdot \alpha^-]\in H\pi(\alpha,U)\cap H\pi(\alpha\cdot \beta,V)=H$. Hence, $\widetilde{x}=\left<\alpha\cdot \gamma\cdot \beta\right>=\left<\alpha\cdot \beta\right>$.
\end{proof}

\begin{remark} \label{enough} In order to apply Lemma~\ref{discreteEQ}, there is no need to check every path $\alpha$: if $\alpha$ and $\alpha'$ are two paths  from $x_0$ to $x$ such that $[\alpha'\cdot \alpha^-]H=H[\alpha'\cdot \alpha^-]$, then we have $H\pi(\alpha,U)\cap H\pi(\alpha\cdot \beta,V)=H$ if and only if $H\pi(\alpha',U)\cap H\pi(\alpha'\cdot \beta,V)=H$.
\end{remark}

\begin{remark}\label{prototypical} The space  $(X,x_0)=(\mathbb{H}\times [0,1],(b_0,0))$ is the prototypical example of a space that does not have the discrete
monodromy property, although $p:\widetilde{X}\rightarrow X$ has unique path lifting. Observe that for the path $\beta(t)=(b_0,t)$ from $x=(b_0,0)$ to $y=(b_0,1)$, the graph of $\Phi_\beta:p^{-1}(x)\rightarrow p^{-1}(y)$ is not discrete.
\end{remark}

\begin{definition}[Locally quasinormal \cite{FGa}] A subgroup $H\leqslant \pi_1(X,x_0)$ is called {\em locally quasinormal} if for every $x\in X$, for every path $\alpha$ in $X$ from $\alpha(0)=x_0$ to $\alpha(1)=x$, and for every $U\in {\mathcal T}_x$, there is a $V\in {\mathcal T}_x$ such that $x\in V\subseteq U$ and $H\pi(\alpha,V)=\pi(\alpha,V)H$.
\end{definition}

\begin{remark} Clearly, every normal subgroup of $\pi_1(X,x_0)$ is locally quasinormal. Combining \cite[Lemma~5.2]{FZ2013} with Remark~\ref{ID}(c), we see that if $X$ is locally path connected, then every open subgroup of $\pi_1(X,x_0)$ (in the topology of Remark~\ref{T1}) is locally quasinormal. For example, the nontrivial subgroup $K\leqslant \pi_1(\mathbb{H},b_0)$ from \cite{FZ2013} is open, while it does not contain any nontrivial normal subgroup of  $\pi_1(\mathbb{H},b_0)$.
\end{remark}

The following is a straightforward variation on \cite[Lemma~3.2]{FGa}:

\begin{lemma}\label{LQN-NBH}
Let $H\leqslant \pi_1(X,x_0)$, $x\in X$, $\alpha$ be a path in $X$ from $x_0$ to $x$, and $U\in {\mathcal T}_x$.
Then the following are equivalent:
\begin{itemize}
\item[(a)] $H\pi(\alpha,U)=\pi(\alpha,U)H$.
\item[(b)] For every $\Phi_\beta:p_H^{-1}(x)\rightarrow p_H^{-1}(x)$ with $\Phi_\beta(\left<\alpha\right>)\in \left<\alpha,U\right>\cap p_H^{-1}(x)$,  we have $\Phi_{\beta}(\left<\alpha,U\right>\cap p_H^{-1}(x))\subseteq \left<\alpha,U\right>\cap p_H^{-1}(x)$.
\end{itemize}
\end{lemma}

We include the  proof for completeness.

\begin{proof}
(i) First, assume that $H\pi(\alpha,U)=\pi(\alpha,U)H$ and $\Phi_\beta(\left<\alpha\right>)\in \left<\alpha,U\right>\cap p_H^{-1}(x)$. Then $\left<\alpha\cdot \beta\right>=\Phi_\beta(\left<\alpha\right>)=\left<\alpha\cdot \delta\right>$ for some loop $\delta$ in $U$.
So, $[\beta]=[\alpha^-]h[\alpha\cdot \delta]$ for some $h\in H$. Let $\left<\gamma\right>\in \left<\alpha,U\right>\cap p_H^{-1}(x)$. Then $[\gamma]=h'[\alpha\cdot \delta']$ for some $h'\in H$ and some loop $\delta'$ in $U$. Since  $[\alpha\cdot \delta'\cdot  \alpha^-]h\in \pi(\alpha,U)H=H\pi(\alpha,U)$,  we have $[\alpha\cdot \delta'\cdot  \alpha^-]h=h''[\alpha\cdot \delta''\cdot  \alpha^-]$ for some $h''\in H$ and some loop $\delta''$ in $U$. Therefore, we have $[\gamma\cdot \beta]=h'[\alpha\cdot \delta'\cdot\alpha^-]h[\alpha\cdot \delta]=h'h''[\alpha\cdot \delta''\cdot \alpha^-][\alpha\cdot \delta]=h'h''[\alpha\cdot \delta''\cdot \delta]$. Hence, $\Phi_\beta(\left<\gamma\right>)=\left<\gamma\cdot \beta\right>=\left<\alpha\cdot \delta''\cdot\delta\right>\in \left<\alpha,U\right>\cap p_H^{-1}(x)$.

(ii) Now, assume that $\Phi_{\beta}(\left<\alpha,U\right>\cap p_H^{-1}(x))\subseteq \left<\alpha,U\right>\cap p_H^{-1}(x)$ whenever $\Phi_\beta(\left<\alpha\right>)\in \left<\alpha,U\right>\cap p_H^{-1}(x)$. It suffices to show that $\pi(\alpha,U)H\subseteq H\pi(\alpha,U)$.
Let $[\tau]\in \pi(\alpha,U)H$. Then $[\tau]=[\alpha\cdot \delta\cdot \alpha^-][\gamma]$ for some loop $\delta$ in $U$ and some $[\gamma]\in H$. Put $\beta=\alpha^-\cdot \gamma\cdot \alpha$. Then $\Phi_\beta(\left<\alpha\right>)=\left<\alpha\right>$. Hence, $\left<\alpha\cdot \delta\cdot \beta\right>=\Phi_\beta(\left<\alpha\cdot \delta\right>)=\left<\alpha\cdot \delta'\right>$ for some loop $\delta'$ in $U$. Therefore, $[\tau]=[\alpha\cdot \delta \cdot \beta \cdot (\delta')^-\cdot \alpha^- ][\alpha\cdot \delta'\cdot \alpha^-]\in H\pi(\alpha,U).$
\end{proof}

\begin{proposition}\label{DMP->HH} Let $H\leqslant \pi_1(X,x_0)$ be  locally quasinormal.
 If $X$ has the discrete monodromy property relative to $H$, then every fiber of $p_H:\widetilde{X}_H\rightarrow X$ is T$_1$.
\end{proposition}

\begin{proof} Suppose there is an $x\in X$ such that $p_H^{-1}(x)$  is not T$_1$. Then there are $\left<\alpha\right>,\left<\gamma\right>\in p_H^{-1}(x)$ with $\left<\alpha\right>\not=\left<\gamma\right>$  such that for every $W\in {\mathcal T}_x$ we have $\left<\gamma\right>\in \left<\alpha,W\right>\cap p_H^{-1}(x)$. Let any $U\in {\mathcal T}_x$ be given. Choose $V\in {\mathcal T}_x$ with $x\in V\subseteq U$ and $H\pi(\alpha,V)=\pi(\alpha,V)H$. Put $\widetilde{V}=\left<\alpha,V\right>\cap p_H^{-1}(x)$ and $\beta=\alpha^-\cdot \gamma$. Then $\Phi_\beta(\left<\alpha\right>)=\left<\gamma\right>\in \widetilde{V}$. By Lemma~\ref{LQN-NBH}, $\Phi_\beta(\left<\gamma\right>)\in \widetilde{V}$. Hence, $(\left<\alpha\right>,\Phi_\beta(\left<\alpha\right>))\not=(\left<\gamma\right>,\Phi_\beta(\left<\gamma\right>))$ are both elements of $\widetilde{V}\times \widetilde{V}$. We conclude that  $\Phi_\beta:p_H^{-1}(x)\rightarrow p_H^{-1}(x)$ is not the identity function and that its graph is not discrete.
\end{proof}

\begin{corollary}
If $X$ has the discrete monodromy property, then $X$ is homotopically Hausdorff.
\end{corollary}

\begin{proof}
Since $H=\{1\}\leqslant \pi_1(X,x_0)$ is locally quasinormal, this follows from Proposition~\ref{DMP->HH}.
\end{proof}

The proof of the following proposition is modelled on \cite{E2002} and \cite{CK}.

\begin{proposition} \label{1dm} All one-dimensional metric spaces and all planar spaces have the discrete monodromy property.
\end{proposition}

\begin{proof}
Let $x,y\in X$ and let $\beta:[0,1]\rightarrow X$ be a path from $\beta(0)=x$ to $\beta(1)=y$.
If $\beta$ is a loop, we assume that it is essential. Let $\alpha$ be any path from $x_0$ to $x$. We wish to find open neighborhoods $U$ and $V$ of $x$ and $y$, respectively, such that
$\pi(\alpha,U)\cap \pi(\alpha\cdot \beta,V)=\{1\}\leqslant \pi_1(X,x_0)$.

(a) Suppose $X\subseteq \mathbb{R}^2$. If $x=y$, choose $\epsilon>0$ such that the loop $\beta$ cannot be homotoped within $X$ into $N_\epsilon(x)=\{z\in \mathbb{R}^2\mid \|x-z\|<\epsilon\}$, relative to its endpoints.
(Here we use the fact that $X$ is homotopically Hausdorff; see Remark~\ref{HH}.) If $x\not=y$, choose any $\epsilon$ with $0<\epsilon<\|x-y\|/2$.

Suppose that $\partial N_r(x)\subseteq X$ for all $0<r<\epsilon$. Then $N_\epsilon(x)\subseteq X$.
In this case, taking $U=N_\epsilon(x)$ gives $\pi(\alpha,U)=\{1\}$.
So, by making $\epsilon$ smaller, if necessary, we may assume that $X\cap \partial N_\epsilon(x)\not=\partial N_\epsilon(x)$.

Put $U=X\cap N_\epsilon(x)$ and $V=X\cap N_\epsilon(y)$. Suppose, to the contrary, that there are essential loops $\delta$ and $\tau$ in $U$ and $V$,
respectively, with $[\alpha\cdot \delta\cdot \alpha^-]=[\alpha\cdot \beta \cdot \tau \cdot \beta^-\cdot \alpha^-]$, i.e., $[\delta]=[\beta\cdot \tau \cdot \beta^-]$. Then there is a map $h:A\rightarrow X$ from an annulus $A$ whose
boundary components $J_1$ and $J_2$ map to $\delta$ and $\tau$, respectively, along with a diametrical arc $a\subseteq A$ connecting $J_1$ to $J_2$ that maps to $\beta$.

If $x\not=y$, then $h^{-1}(X\cap \partial N_\epsilon(x))$ clearly separates $J_1$ from $J_2$ in $A$.
However, this is also true if $x=y$; for otherwise there is an arc $a'\subseteq A$ connecting $J_1$ to $J_2$ which $h$ maps to a path $\beta'$ in $U$, so that $\beta$ is homotopic within $X$, relative to its endpoints, to the concatenation of an initial subpath $\delta'$ of $\delta$, the path $\beta'$, a terminal subpath $\tau'$ of $\tau$, and a path of the form $\tau\cdot \tau\cdots  \tau$ or $\tau^-\cdot \tau^- \cdots  \tau^-$,  all of which lie in $U$, violating the choice of $\epsilon$.
Therefore, as in the proof of \cite[Lemma 5.5]{CK}, the loop $\delta$ contracts within $X$; a contradiction.

(b) Suppose $X$ is a one-dimensional metric space. We may assume that $\beta$ is a reduced non-degenerate path (possibly a loop). Choose open neighborhoods $U$ and $V$ of $x$ and $y$ in $X$, respectively, such that $\beta$ is not contained in $U\cup V$. Suppose, to the contrary, that there are essential loops $\delta$ and $\tau$ in $U$ and $V$,
respectively, with $[\delta]=[\beta\cdot \tau \cdot \beta^-]$. We may assume that both $\delta$ and $\tau$ are reduced. Then $\beta([0,1])\subseteq \delta([0,1])\cup \tau([0,1])\subseteq U\cup V$ (see Lemma~\ref{cancellinglemma}); a contradiction.
\end{proof}

\begin{theorem}\label{PhasDMP}
$\mathbb{P}$ has the discrete monodromy property.
\end{theorem}

\begin{proof}
 Let $1\not=[\beta]\in\pi_1(\mathbb{P},b_0)$. In view of Remark~\ref{ID}(a), Lemma~\ref{discreteEQ} and Remark~\ref{enough}, and since $\mathbb{P}$ locally contractible at every point other than $b_0$, it suffices to find an open neighborhood $V$ of $b_0$ in $\mathbb{P}$ such that $\pi(c_{b_0},V)\cap \pi(\beta,V)=\{1\}$, where $c_{b_0}$ denotes the constant path at $b_0$.

Choose $n\in \mathbb{N}$ such that $\beta([0,1])\cap P^\circ_i=\emptyset$ for all $i>n$. Put $\beta'=d_n\circ d_{n-1}\circ \cdots \circ d_1\circ \beta$. Then $\beta'([0,1])\subseteq \mathbb{H}^+_{n+1}$ and $[\beta]=[\beta']\in\pi_1(\mathbb{P},b_0)$. So, we may assume that $\beta$ is a reduced loop in $\mathbb{H}^+_{n+1}$. Increasing $n$ if necessary, we may assume that $\beta$ traverses one of the circles $B_{i,j}$ with $i\in \{1,2,\dots, n\}$ and $j\in \{1,2\}$. As in the proof of Theorem~\ref{UVThm}, we may construct an open neighborhood $V$ of $b_0$ in $\mathbb{P}$ that does not fully contain $B_{i,j}$, such that $V$ deformation retracts onto $\mathbb{H}_{n+1}\subseteq \mathbb{H}$.

Suppose, to the contrary, that there are essential loops $\delta$ and $\tau$ in $V$ such that $[\delta]=[\beta\cdot \tau\cdot \beta^-]\in\pi_1(\mathbb{P},b_0)$. We may assume that both $\delta$ and $\tau$ are reduced loops in $\mathbb{H}_{n+1}$. Let $F$ be a homotopy from $\delta$ to $\beta\cdot \tau\cdot \beta^-$ (relative to endpoints) within $\mathbb{P}$. Choose $k\geqslant n$ such that the image of $d_k\circ d_{k-1}\circ \cdots \circ d_1\circ F$ is contained in $\mathbb{H}^+_{k+1}$. Let $\beta'$, $\delta'$ and $\tau'$ be the composition of  $d_k\circ d_{k-1}\circ \cdots \circ d_1$ with  $\beta$, $\delta$ and $\tau$, respectively. Then $\beta'$, $\delta'$ and $\tau'$ are reduced loops in $\mathbb{H}^+_{k+1}$ such that $[\delta']=[\beta'\cdot \tau'\cdot \beta'^-]\in\pi_1(\mathbb{H}^+_{k+1},b_0)$. However, neither $\delta'$ nor $\tau'$ traverses $B_{i,j}$, while $\beta'$ does; a contradiction. (See Lemma~\ref{cancellinglemma}.)
\end{proof}

 Consider the subspace $w(X)$ of ``wild'' points of $X$, defined by \[w(X)=\{x\in X\mid X \mbox{ is not semilocally simply connected at } x\}.\]
The following is the main utility for spaces satisfying the discrete monodromy property, as implicitly used in \cite{E2002} and \cite{CK}. The proof is given after Corollary~\ref{hmpty} below.

\begin{theorem}\label{utility} Suppose both $X$ and $Y$ have the discrete monodromy property. If $f:X\rightarrow Y$ is a homotopy equivalence with homotopy inverse $g:Y\rightarrow X$, then $f$ maps $w(X)$ homeomorphically onto $w(Y)$, with inverse $g|_{w(Y)}$.
\end{theorem}

\begin{remark}
 In order to see the necessity of the assumptions in Theorem~\ref{utility}, consider $X=\mathbb{H}$ and $Y=\mathbb{H}\times [0,1]$. Then $X$ has the discrete monodromy property, $X$ and $Y$ are homotopy equivalent, but $w(X)=\{b_0\}$ and $w(Y)=\{b_0\}\times [0,1]$ are not homeomorphic.
\end{remark}

For a path $\beta:[0,1]\rightarrow X$, we let  $\varphi_\beta:\pi_1(X,\beta(0))\rightarrow \pi_1(X,\beta(1))$ be the base point
changing isomorphism defined by $\varphi_\beta([\delta])=[\beta^-\cdot \delta\cdot \beta]$.

\begin{lemma}\label{agree} Suppose $Y$ has the discrete monodromy property.
Let $f,g:X\rightarrow Y$ be maps such that $\varphi_\beta\circ f_\#= g_\#:\pi_1(X,x)\rightarrow \pi_1(Y,g(x))$ for some $x\in X$ and some path $\beta$ in $Y$ from $f(x)$ to $g(x)$.
If $f(x)\not=g(x)$, then there is a $W\in {\mathcal T}_x$ such that $f_\#:\pi_1(W,x)\rightarrow \pi_1(Y,f(x))$ is  trivial.
\end{lemma}

\begin{proof} Suppose $f(x)\not=g(x)$. Let $c_{f(x)}$ be the constant path at $f(x)$. By Lemma~\ref{discreteEQ}, there are $U\in {\mathcal T}_{f(x)}$ and $V\in {\mathcal T}_{g(x)}$ such that $\pi(c_{f(x)},U)\cap \pi(\beta,V)=\{1\}\leqslant \pi_1(Y,f(x))$. Choose $W\in {\mathcal T}_x$ with $f(W)\subseteq U$ and $g(W)\subseteq V$. Let $\ell$ be a loop in $W$, based at $x$. Then $f_\#([\ell])=[f\circ\ell]=[\beta\cdot (g\circ\ell)\cdot \beta^-]\in \pi(c_{f(x)},U)\cap \pi(\beta,V)=\{1\}$.
\end{proof}

\begin{corollary}\label{hmpty} Suppose $X$ is path connected and  $Y$ has the discrete monodromy property.
If $f,g:X\rightarrow Y$ are homotopic maps and $f_\#:\pi_1(X,x_0)\rightarrow \pi_1(Y,f(x_0))$ is injective, then $f|_{w(X)}=g|_{w{(X)}}$.
\end{corollary}

\begin{proof}
Let $F:X\times [0,1]\rightarrow Y$ be a map with $F(x,0)=f(x)$ and $F(x,1)=g(x)$ for all $x\in X$. Fix $x\in w(X)$ and let $\beta:[0,1]\rightarrow Y$ be given by $\beta(t)=F(x,t)$. Then $\varphi_\beta\circ f_\#= g_\#:\pi_1(X,x)\rightarrow \pi_1(Y,g(x))$. Since $x\in w(X)$ and since $f_\#:\pi_1(X,x)\rightarrow \pi_1(Y,f(x))$ is injective, its follows from Lemma~\ref{agree} that $f(x)=g(x)$.
\end{proof}

\begin{proof}[Proof of Theorem~\ref{utility}]
Let $f:X\rightarrow Y$ and $g:Y\rightarrow X$ be a pair of homotopy inverses. Then, for every $x\in X$, $f_\#:\pi_1(X,x)\rightarrow \pi_1(Y,f(x))$ is an isomorphism; in particular, it is injective. Therefore, $f(w(X))\subseteq w(Y)$. Similarly, $g(w(Y))\subseteq w(X)$. Since $g\circ f$ is homotopic to the identity it follows from Corollary~\ref{hmpty} that $g(f(x))=x$ for all $x\in w(X)$. Similarly, $f(g(y))=y$ for all $y\in w(Y)$.
\end{proof}

When working with spaces for which homomorphisms between fundamental groups are induced by continuous maps up to base point change, as is the case among all one-dimensional and planar Peano continua \cite{CK, E2010, K}, the following provides additional utility:

\begin{corollary}
Suppose both $X$ and $Y$ have the discrete monodromy property.\linebreak Let $\phi:\pi_1(X,x_0)\rightarrow \pi_1(Y,y_0)$ be an isomorphism with $\phi=\varphi_\alpha\circ f_\#$ and $\phi^{-1}=\varphi_\beta\circ g_\#$ for some maps $f:X\rightarrow Y$ and $g:Y\rightarrow X$ and some  paths $\alpha$ and $\beta$.
Then $f$ maps $w(X)$ homeomorphically onto $w(Y)$, with inverse $g|_{w(Y)}$.
\end{corollary}

\begin{proof} For every $x\in X$, $f_\#:\pi_1(X,x)\rightarrow \pi_1(Y,f(x))$ is injective and for every $y\in Y$, $g_\#:\pi_1(Y,y)\rightarrow \pi_1(X,g(y))$ is injective. Therefore, $f(w(X))\subseteq w(Y)$ and $g(w(Y))\subseteq w(X)$. Let $x\in w(X)$. Choose a path $\gamma$ in $X$ from $x_0$ to $x$. Put $\delta=(g\circ f\circ \gamma)^-\cdot (g\circ \alpha)\cdot \beta \cdot \gamma$.
 Since $\varphi_{(g\circ \alpha)\cdot \beta}\circ (g\circ f)_\#=\varphi_\beta\circ g_\#\circ \varphi_\alpha \circ f_\#=id:\pi_1(X,x_0)\rightarrow \pi_1(X,x_0)$, we have $\varphi_\delta\circ (g\circ f)_\#=id:\pi_1(X,x)\rightarrow  \pi_1(X,x)$. By Lemma~\ref{agree}, $g(f(x))=x$. Similarly, $f(g(y))=y$ for all $y\in w(Y)$.
\end{proof}

\end{document}